\documentclass[12pt, oneside, reqno]{amsart}

\usepackage[margin=2.5cm]{geometry}
\usepackage{amssymb,amsmath,amsthm,amsfonts}
\usepackage{mathtools, enumitem}
\usepackage[bbgreekl]{mathbbol}
\usepackage[cal=boondox]{mathalpha}
\usepackage[colorlinks=true,linkcolor=blue,citecolor=blue]{hyperref}
\usepackage{todonotes}
\DeclareFontFamily{U}{matha}{\hyphenchar\font45}
\DeclareFontShape{U}{matha}{m}{n}{
  <5> <6> <7> <8> <9> <10> gen * matha
  <10.95> matha10 <12> <14.4> <17.28> <20.74> <24.88> matha12
}{}
\DeclareSymbolFont{matha}{U}{matha}{m}{n}
\DeclareFontSubstitution{U}{matha}{m}{n}
\DeclareFontFamily{U}{mathx}{\hyphenchar\font45}
\DeclareFontShape{U}{mathx}{m}{n}{
  <5> <6> <7> <8> <9> <10>
  <10.95> <12> <14.4> <17.28> <20.74> <24.88> mathx10
}{}
\DeclareSymbolFont{mathx}{U}{mathx}{m}{n}
\DeclareFontSubstitution{U}{mathx}{m}{n}
\DeclareMathDelimiter{\vvvert}{0}{matha}{"7E}{mathx}{"17}

\newcommand{\clush}{\cl{u}_h^{(\sigma)}}
\newcommand{\uhs}{{u_h^{(\sigma)}}}
\newcommand{\muhs}{{\mu_h^{(\sigma)}}}
\newcommand{\etas}{\eta^{(\sigma)}}
\newcommand{\oscs}{\osc^{(\sigma)}}

\newcommand{\cl}[1]{\mathcal{{#1}}}
\newcommand{\R}{\mathbb{R}}
\newcommand{\Om}{\varOmega}
\newcommand{\oh}{\varOmega_h}

\newcommand{\Khat}{\hat K}
\newcommand{\Sc}{\mathcal{S}}
\newcommand{\bdr}{b_{s}}            
\newcommand{\hbdr}{\hat{b}_{s}}
\newcommand{\bdrb}{b_{s(\beta)}}
\newcommand{\bint}{b}                      

\newcommand{\Pk}{P_k}
\newcommand{\Pkint}[1]{\mathring{P}_{#1}}  
\newcommand{\Pkperp}[1]{P_{#1}^{\perp}}    
\newcommand{\W}[2]{W_{#1}(\partial_\beta, #2)}   
\newcommand{\trc}{\mathop{\gamma}}    

\newcommand{\nrmK}[3]{\| #1 \|_{L_{#2}(#3)}}
\newcommand{\ntrip}[1]{\vvvert {#1} \vvvert}
\newcommand{\tnrm}[2]{\ntrip{#1}_{#2}}
\newcommand{\ip}[3]{(#1, #2)_{#3}}         
\newcommand{\ap}[1]{\langle {#1} \rangle}
\newcommand{\Dh}{D_h}
\DeclareMathOperator{\sgn}{sgn}
\DeclareMathOperator{\dive}{div}
\DeclareMathOperator{\osc}{osc}
\newcommand{\intr}[1]{\mathop{\mathrm{int}}{#1}}
\newcommand{\lsim}{\,\lesssim\,}
\newcommand{\Xhok}{\hat{X}_{h,0}^k}
\newcommand{\Xhk}{\hat{X}_h^k}
\newcommand{\ET}{E}                    
\newcommand{\bform}{\cl b}       
\newcommand{\Enrm}{\cl E}              
\newtheorem{theorem}{Theorem}[section]
\newtheorem{lemma}[theorem]{Lemma}
\newtheorem{proposition}[theorem]{Proposition}
\newtheorem{corollary}[theorem]{Corollary}
\newtheorem{assumption}[theorem]{Assumption}
\theoremstyle{definition}
\newtheorem{definition}[theorem]{Definition}
\theoremstyle{remark}
\newtheorem{remark}[theorem]{Remark}

\title[Fortin operators for DPG advection]{Fortin operators for
  {DPG} advection discretizations}

\author{Pablo Cort{\'e}s Castillo}
\address{Portland State University, PO Box 751, Portland OR 97201, USA }
\email{pcortes@pdx.edu}

\author{Jay Gopalakrishnan}
\address{Portland State University, PO Box 751, Portland OR 97201, USA }
\email{gjay@pdx.edu}

\begin{document}

\begin{abstract}
  We construct Fortin operators for discontinuous Petrov--Galerkin
  (DPG) discretizations of the advection equation with a piecewise
  constant divergence-free advection vector~$\beta$, on simplicial
  meshes of any spatial dimension~$N$ and for any polynomial
  degree~$k \ge 1$ of the trial space. A minimal test space is built on
  each element from facet and interior bubbles and later augmented. The Fortin operator is shown
  to be bounded, uniformly over shape-regular mesh families, in the
  natural $\beta$-weighted broken test graph norm built on $L_q$ for
  every $1 < q < \infty$, where $q$ is the exponent conjugate to the
  trial exponent $p$. A non-characteristic facet condition is
  assumed when $k \ge 2$, while
  the lowest-order case requires no such condition and admits
  characteristic facets. As applications we prove that the practical
  fully discrete residual minimization method is quasioptimal in the
  DPG energy norm for every $1 < p < \infty$, with a quasioptimality
  constant governed solely by the Fortin operator, that its
  computable residual is a globally reliable and efficient a
  posteriori error estimator, and that augmenting the test space and changing the test norm improves the results in  the lowest-order case.
\end{abstract}

\maketitle

\section{Introduction}
\label{sec:intro}

The discontinuous Petrov--Galerkin (DPG) method
\cite{DemkoGopal10,DemkoGopal11,DPGacta} is a residual minimization
method. Given a variational formulation  over a
trial space $X$ and a test space $Y$, one selects the discrete
solution by minimizing the residual over a
finite dimensional trial subspace $X_h \subset X$, measuring it in the dual norm of
$Y$. So posed, the method is immediately stable and quasioptimal.  The
difficulty is that the dual norm computation requires a solve that is both global and
infinite-dimensional. Variational DPG reformulation techniques (by breaking the test space~\cite{CarstDemkoGopal16, DPGacta}) have shown how to replace the global solves by local
element-by-element solves. The infinite-dimensionality is then eliminated
by the \emph{practical DPG method} which measures the residual in the
local dual norm of a finite dimensional test space $Y_h \subset Y$
instead.  The application of this method to advection is
studied in this paper.  The standard device~\cite{FuhreHeuer24,GopalQiu14} for studying the price of
replacing $Y$ by $Y_h$ is a \emph{Fortin operator}: a linear
map $\varPi_h : Y \to Y_h$, bounded independently of the mesh,  whose deviation $v - \varPi_h v$ is
not seen  by the discrete trial space on the mesh.
For advection, however, no such operator into a
test space of polynomials on the same mesh elements has been available.
We construct one in this paper.

We consider the model problem
\begin{equation}
  \label{eq:advection}
  \beta \cdot \nabla u  = f \quad\text{ in } \Om,
  \qquad
  u = 0 \quad\text{ on } \Gamma_{\mathrm{in}},
\end{equation}
where $\Om \subset \R^N$ is a Lipschitz polyhedral domain,
$\Gamma_{\mathrm{in}} := \{ x \in \partial\Om : \beta(x) \cdot n(x)
< 0\}$ is the inflow boundary, and $\beta$ is a piecewise constant,
divergence-free advection field. For such problems, it is natural to adopt  a functional setting beyond
Hilbert spaces. Residual minimization in $L_p$-type graph norms is
attractive for transport because it can reduce the over-smearing
associated with $L_2$ settings and, as $p \to 1^+$, eliminate Gibbs
phenomena near discontinuities~\cite{BroerDahmeSteve18,Guerm99,MugaTylerZee19,MugaZee20}.
Such a
residual minimization for advection--reaction in $L_p(\Om)$
was carried out in~\cite{MugaTylerZee19},
working with globally conforming test spaces.
We adopt this avenue, but with broken test spaces.
Our trial space pairs discontinuous
piecewise polynomials of degree $k-1$ with interface variables of
degree $k$, for arbitrary $k \ge 1$. Our minimal local test space on an
$N$-simplex $K$ is
$
  V(K) \;=\;  \bdr \, \Pkperp{k}(K) \;+\; \bint\, P_{k-2}(K).
$
Here $\bdr$ is a signed sum of facet bubbles tracking the
inflow/outflow pattern of $\beta$, $\bint$ is the interior bubble,
and $\Pkperp{k}(K)$ excludes the degree-$k$ interior bubbles (all precisely defined later). This
space has the minimal dimension matching the active trial degrees of
freedom, and the resulting Fortin operator is bounded uniformly in
the mesh.

The main results have three parts. First, we construct this minimal
Fortin operator for general $k$ and $N$ and prove its mesh-uniform
boundedness in a $\beta$-weighted broken graph norm for every
$1<p<\infty$. Our lowest-order ($k=1$) results are for any shape-regular mesh,
while for $k\ge2$ we assume  a uniform non-characteristic facet
condition. Second, the practical method is quasioptimal in
the mesh-dependent DPG energy norm, converges at the expected rate in that norm, and admits a
globally reliable and efficient residual estimator. Finally, we
introduce a scale $\sigma$ into the test norm and show that a
mean-preserving augmentation of the test space, costing at most two
functions per element, makes the Fortin operator bounded uniformly
for every fixed $\sigma$ above the intrinsic element-wise scale. At
lowest order this yields optimal first-order convergence in the
mesh-independent $L_2(\Om)$ norm and direct estimator control up to
a data oscillation, which vanishes for piecewise constant data (Theorem~\ref{thm:kone-rates}). To
our knowledge this is the first analysis exhibiting the influence of
such a scale on the error of any practical DPG method.

The closest antecedent of our test space is in the transport method
of~\cite{DemkoGopal10}, one of the origins of the modern  DPG methodology.
Optimal $L_2$ estimates were obtained for constant $\beta$ using exactly
tailored, flow-aligned test functions that are generally nonpolynomial and
may be discontinuous within an element. Its two-dimensional error analysis
relies on a flow-dependent mesh-layer condition. The present work instead
constructs the  explicit minimal polynomial test space $V(K)$ in each
$N$-simplex with a Fortin bound  for every $1<p<\infty$, at the cost of
weighted test norms and, for $k\ge2$, a non-characteristic facet condition.
The space $V(K)$ above may be viewed as a polynomial surrogate, licensed
by the Fortin condition: its two summands play the roles of the trace  and
interior parts of the earlier space.

Our results complement those of~\cite{BroerDahmeSteve18}. In the
Hilbert setting and for variable convection fields, they  achieve
discrete inf-sup stability
uniformly in the relative orientation of the flow and mesh,
in a mesh-independent graph norm,
by taking test polynomials of one degree higher on a subgrid
refinement of each element, of fixed but unspecified depth.
Our construction instead provides an explicit minimal test space of
degree $k+N$ on the original mesh, together with a Fortin operator
and a posteriori estimates for every $1<p<\infty$.
The mesh-uniform Fortin bound we provide uses, in contrast, natural
element-wise weights, and we recover mesh-independent optimality only
at lowest order. Furthermore, while their analysis also accommodates
variable convection fields, extending our scaling argument beyond
piecewise constant fields remains open.

This distinction also clarifies the conclusions of~\cite{DemkoRoberMunoz22}. For an
element-wise constant convection field, their attempted local Fortin
construction for the convection--reaction problem could not be
bounded uniformly in the unweighted graph norm, particularly as
element facets align with the convection direction. For pure
advection, we show that a uniformly bounded minimal local operator
does exist when the test norm carries certain weights (called $a_K$ later). The angle
sensitivity is nevertheless genuine for $k\ge2$. Somewhat surprisingly, at lowest order, our estimates are completely insensitive to the flow angle and to the presence of
characteristic facets.

The rest of the paper is organized as follows.
Section~\ref{sec:weakform} fixes notation and states the practical
DPG method. Section~\ref{sec:testspace} constructs the local test
space and Section~\ref{sec:fortin} the Fortin operator, whose
unisolvency is proved there. Section~\ref{sec:scaling} develops the
reference-element equivalences and scaling identities, which
Section~\ref{sec:continuity} uses to prove boundedness of
$\varPi_h$ and of its augmented variants.
Section~\ref{sec:quasiopt} places the method in a functional
setting, recalls the wellposedness of the undiscretized problem, and
proves quasioptimality and convergence rates.
Section~\ref{sec:apost} treats a posteriori error control and
Section~\ref{sec:lowest} the lowest-order case with a scaled test
norm. Section~\ref{sec:remarks} collects concluding remarks and open
problems.

\section{The DPG method for advection}
\label{sec:weakform}

Let $\oh$ be a conforming simplicial mesh of $\Om \subset \R^N$,
$N \ge 1$. Let $\Khat$ denote the unit reference $N$-simplex with
vertices $\hat a_0, \dots, \hat a_N$ and barycentric coordinates
$\hat\lambda_0, \dots, \hat\lambda_N$. Each element $K \in \oh$, with
vertices $a_0,\dots,a_N$, is the image of $\Khat$ under an affine
homeomorphism
\begin{equation}
  \label{eq:affinemap}
  x = \Phi_K(\hat x) = M_K \hat x + g_K,
  \qquad M_K \in \R^{N\times N},\ g_K \in \R^N,
\end{equation}
mapping $\hat a_i$ to $a_i$. 
Let
$h_K = \operatorname{diam} K$.
For a measurable set $D$, $|D|$ denotes its measure, and  $|\cdot |$ also
denotes the Euclidean norm on $\R^N$ and the associated matrix norm, when applied to a vector or matrix;
the meaning will be clear from context. Recall the standard relations
\begin{equation}
  \label{eq:MhK}
  |M_K| \,\le\, C \, h_K, \qquad h_K \,\le\, C\, |M_K|,
\end{equation}
with constants depending only on $\Khat$. Throughout this paper,
we consider families of
meshes that are \emph{shape regular}:
\begin{equation}
  \label{eq:shapereg}
  \kappa_K := |M_K|\, |M_K^{-1}| \,\le\, \kappa_0
  \qquad \text{for all } K \in \oh,
\end{equation}
for a fixed $\kappa_0$.

The graph space for the advection operator $\beta \cdot \nabla$, and a
``broken'' version of it, are defined, for $1 \le r \le \infty$ and any
(sub)domain $D \subseteq \Om$, by
\begin{equation*}
  \W{r}{D} := \{ v \in L_r(D) \,:\, \beta\cdot\nabla v \in L_r(D) \},
  \qquad
  \W{r}{\oh} := \bigtimes_{K \in \oh} \W{r}{K}.
\end{equation*}
For an integer $s \ge 0$ and $1 \le r \le \infty$, we denote by
$W_r^s(D)$ the standard Sobolev space of functions on $D$ whose
distributional derivatives up to order $s$ lie in $L_r(D)$, with norm
$\| \cdot \|_{W_r^s(D)}$ and seminorm
$| v |_{W_r^s(D)} := \big( \sum_{|\alpha| = s}
\| \partial^\alpha v \|_{L_r(D)}^r \big)^{1/r}$ (with the usual
modification when $r = \infty$), and we write $H^s(D) := W_2^s(D)$.
Using a prime ($'$) to denote the dual space, and letting
\begin{equation}
  \label{eq:conjugate}
  q = p/(p-1) \text{ when } p>1,
  \quad \text{ and }\quad q = \infty \text{ when } p = 1,
\end{equation}
be  the conjugate exponent of  $p$,
define for each $K\in\oh$,  the element boundary
operator $D_K   : \W{p}{K} \to \big(\W{q}{K}\big)'$,
\begin{equation}
  \label{eq:DK}
  \ap{D_K w, v} \,:=\, \ip{\beta\cdot\nabla w}{v}{K}
  + \ip{w}{\beta\cdot\nabla v}{K},
\end{equation}
where $\ip{\cdot}{\cdot}{K}$ denotes the $L_2(K)$ inner product,
$w \in \W{p}{K}$ and $v \in\W{q}{K}$.  Throughout the action of a
functional $f \in X'$ on an $x \in X$ is denoted by $f(x)$ or
$\ap{f, x}_X$ where the subscript is omitted when no confusion can
arise.

Let $P_k(D)$ denote the space of polynomials of degree at most $k$
restricted to a subdomain $D$ of $\Om$ (with the convention that
$P_k(D) = \{0\}$ when $k<0$) and let
\[
  P_k(\oh) = \bigtimes_{K \in \oh} P_k(K)
\]
denote the space of piecewise  polynomials of degree at most $k$ on $\Om$ (which are possibly discontinuous across mesh facets).
When $w$ and $v$ are smooth functions (or can be approximated
arbitrarily closely by smooth functions) up to (including) the boundary
of $K$, such as $v, w \in P_k(\oh)$, integration by parts and the fact that
$\beta|_K$ is constant shows that
\begin{equation}
  \label{eq:DKbdry}
  \ap{D_K w, v} = \int_{\partial K} (\beta\cdot n)\, w\, v \,ds .
\end{equation}
H\"older's inequality shows directly that \eqref{eq:DK} is continuous
on the indicated graph spaces. Formula \eqref{eq:DKbdry} is its
boundary representation whenever $w$ and $v$ are smooth.

For a fixed degree $k\geq 1$, let
$
  S_h^k := \Pk(\oh) \cap C(\overline\Om)
  $
  denote the standard Lagrange finite element space
  and let
  $S_{h,0}^k:= \{z_h \in S_h^k:
  z_h|_{\Gamma_{\mathrm{in}}}=0\}$.
Then  any $z_h \in S_h^k$ has  single-valued
traces on mesh interfaces.
Abbreviate
\begin{equation}
  \label{eq:trfunctional}
  \ap{\Dh z, v}
  := \sum_{K \in \oh} \ap{D_K(z|_K),v|_K},
  \qquad v \in Y:= \W{q}{\oh}.
\end{equation}
(For now, this defines a linear operator $D_h : S_h^k \to Y'$, but we will extend its domain later.)
Write $\trc$ for the element-wise trace, i.e., $\trc(z_h)|_{\partial K}  = z_h|_{\partial K}$.
Define the discrete interface spaces
\begin{equation}
  \label{eq:trial}
  \Xhk := \Dh(S_h^k),
  \qquad
  \Xhok := \Dh(S_{h,0}^k).
\end{equation}
The DPG method uses the trial space
\begin{equation}
  \label{eq:Xh}
  X_h := P_{k-1}(\oh) \times \Xhok,
\end{equation}
whose second component incorporates the inflow boundary condition,
together with a finite dimensional
discrete test space $Y_h \subset P_r(\oh) \subset \W{q}{\oh}$ for some degree $r$ to be determined.
Multiplying \eqref{eq:advection} by a test function $v$ element-wise
and integrating by parts gives the bilinear form
\begin{equation}
  \label{eq:bform}
  \bform\big( (w_h, \Dh z_h),\, v \big)
  := \sum_{K \in \oh}\left[ -\int_K w_h \,\beta\cdot\nabla v \,dx
    + \int_{\partial K} \beta \cdot n\, z_h \,v \,ds \right]
\end{equation}
for any $w_h \in P_{k-1}(\oh)$, $z_h \in S_h^k$ and $v \in Y_h$.
Here we have used~\eqref{eq:DKbdry}, which also
shows that $\ap{\Dh z_h, v}$ depends only on $\trc(z_h)$.
The same formula~\eqref{eq:DKbdry} shows that, on replacing the second
term of the sum in~\eqref{eq:bform} by $\ap{\Dh z_h, v}$
of~\eqref{eq:trfunctional}, the bilinear form $\bform$ allows for test
functions $v$ in the larger space $\W{q}{\oh}$, this extension is recorded
in~\eqref{eq:extended-b} of Section~\ref{sec:quasiopt}.

Given $1 \le p < \infty$ and data
$f \in L_p(\Om)$, set $\ell(v) := \ip{f}{v}{\Om}$, a linear functional
on $\W{q}{\oh}$. The {\em practical DPG method} then finds
\begin{equation}
  \label{eq:method}
  \cl u_h\, := \, \operatorname*{arg\,min}_{\cl w_h \in X_h}
  \big\| \ell - \bform( \cl w_h, \cdot\,) \big\|_{Y_h'}
\end{equation}
which is the practical DPG solution of the advection
problem. The only requirement we shall place on $Y_h$ is that it
contain an advection-adapted space $V_h$ of minimal dimension,
constructed in Section~\ref{sec:testspace}, on which the Fortin
operator of this paper is built; see \eqref{eq:Yh-condition}.
Specification of the dual norm $\| \cdot \|_{Y_h'}$ and the
space $V_h$ (done later in \eqref{eq:Ynorm} and \eqref{eq:Vh}, respectively), completes the mathematical definition of the
method. The method, as stated in~\eqref{eq:method}, performs a
residual minimization over the discrete trial space.  For $p = 2$,
\eqref{eq:method} is equivalent, as is well known
\cite{DPGacta}, to a linear system computable element by
element (via the discrete Riesz map of the Hilbert subspace $Y_h$ or
via ``discrete optimal test functions''), and also to a saddle point
problem. For $p \neq 2$ it is equivalent to a mixed system with
monotone nonlinearity, in which the Riesz map is replaced by the
duality map of the Banach subspace $Y_h$ as shown in
\cite[eq.~(4.2)]{MugaTylerZee19} and \cite{MugaZee20}, where they also
call it ``DDMRes method.''  We will not need these reformulations in our
analysis. We will  work directly with~\eqref{eq:method}.

\section{The discrete local test space}
\label{sec:testspace}

The test space of the practical DPG method is constructed element by
element.  Hence throughout this section,
we fix a $K \in \oh$.
Using the notation of \eqref{eq:affinemap}, we write
$\lambda_i = \hat\lambda_i \circ \Phi_K^{-1}$ for the barycentric
coordinates of $K$, $F_i$ for the facet of $K$ opposite to $a_i$, and
$n_i$ for the outward unit normal on $F_i$.

\begin{assumption}
  \label{asm:beta}
  The advection field $\beta$ is constant on each element ($\beta|_K
  \equiv \beta_K \in \R^N$ for all $K \in \oh$) with continuous
  normal component across element interfaces (so that
  $\dive \beta = 0$ on $\Om$), and $\beta_K \neq 0$ for every
  $K \in \oh$.
\end{assumption}

Fields satisfying Assumption~\ref{asm:beta} arise canonically as the
lowest-order Raviart--Thomas interpolant of a divergence-free
velocity field.
For $\beta \in \R^N \setminus \{0\}$, define its \emph{sign vector}
(relative to the facets of $K$)
\begin{equation}
  \label{eq:signs}
  s(\beta) \,:=\, (s_0, \dots, s_N), \qquad
  s_i \,:=\, \sgn ( \beta \cdot n_i ) \,\in\, \{-1, 0, +1\} .
\end{equation}
Under Assumption~\ref{asm:beta}, write $s := s(\beta_K)$ for the sign
vector of $\beta_K$ on each $K$ (suppressing its dependence on $K$
from the notation). A facet $F_i$ is called a \emph{characteristic
facet} if $s_i=0$ (and non-characteristic otherwise).

\begin{lemma}
  \label{lem:two-noncharacteristic}
  For any $N$-simplex $K$ and any $\beta \in \R^N \setminus \{0\}$,
  the sign vector $s(\beta)$ of~\eqref{eq:signs} has at least two
  nonzero entries. In particular, under Assumption~\ref{asm:beta},
  every $K \in \oh$ has at least two non-characteristic facets.
\end{lemma}

\begin{proof}
  Any $N$ of the $N+1$ outward normals $n_i$ of a simplex are linearly
  independent. Hence the nonzero vector $\beta$ can be orthogonal
  to at most $N - 1$ of them.
\end{proof}
\begin{definition}[Bubbles]
  \label{def:bubbles}
  Define the interior and facet bubbles of
  $K$ by
  \begin{equation}
    \label{eq:bubbles}
    \bint := \prod_{j=0}^N \lambda_j \in P_{N+1}(K), \qquad
    b_i := \prod_{j \ne i} \lambda_j  = \frac{\bint}{\lambda_i} \in P_N(K).
  \end{equation}
  Note that $\bint > 0$ in the interior of $K$ and $\bint = 0$ on
  $\partial K$, while $b_i$ vanishes on $\partial K$ except in the
  relative interior of $F_i$, denoted by
  $\intr{F}_i = F_i \setminus \partial F_i$, where it is positive.
  Using the signs $s_i$ of \eqref{eq:signs}, set
  \begin{equation}
    \label{eq:bdrbubble}
    \bdr := \sum_{i=0}^{N} s_i\, b_i .
  \end{equation}
  For $j > N$ let
  \begin{align*}
    \Pkint{j}(K)
    & := \{ z \in P_j(K) : z|_{\partial K} = 0 \}
      = \bint \, P_{j - N - 1}(K),
    \\
    \Pkperp{j}(K)
    & :=
    \{
    z \in P_{j}(K)
    \; {:} \;
    ( z, \mathring{p})_K =0 \text{ for all } \mathring p \in \mathring{P}_j(K)
      \}
    = \Pkint{j}(K)^{\perp} \cap P_j(K),
  \end{align*}
  and
  \begin{align}
    \label{eq:jN-pols}
    \Pkint{j}(K) = \{0\}, \qquad
    \Pkperp{j}(K) = P_j(K),
    \qquad
    \text{ when }   0 \le j \le N.
  \end{align}
\end{definition}

When $s=s(\beta)$, two properties of $\bdrb$ are used repeatedly. First,
\begin{equation}
  \label{eq:bdr-pos}
  (\beta \cdot n)\, \bdrb \big|_{F_i}
  = (\beta\cdot n_i)\, s_i\, b_i \big|_{F_i}
  = |\beta \cdot n_i| \; b_i\big|_{F_i} \;\ge 0,
  \qquad i = 0,\dots,N,
\end{equation}
since on $F_i$ all other facet bubbles vanish. Second, $\bdrb$ is
nonzero on $\intr{F}_i$ for every non-characteristic facet $F_i$,
of which there are at least two for the $\beta$ we consider
(see Lemma~\ref{lem:two-noncharacteristic}), so
\begin{equation}
  \label{eq:bdr-nonzero}
  \bdrb \not\equiv 0
\end{equation}
is a nonzero polynomial.

\begin{lemma}[Trace extension]
  \label{lem:extension}
  The trace map $\trc(u) := u|_{\partial K}$ restricted to
  $\Pkperp{k}(K)$ is a bijection onto $\trc(\Pk(K))$.
\end{lemma}

\begin{proof}
  The map $\trc: \Pk(K) \to \trc(\Pk(K))$  has kernel $\Pkint{k}(K)$. Since
  $\Pk(K) = \Pkint{k}(K) \oplus \Pkperp{k}(K)$, the restriction of
  $\trc$ to $\Pkperp{k}(K)$ is injective. It also remains onto, since
  any $u \in \Pk(K)$ splits as $u = \mathring u + u^\perp$ with
  $\mathring u \in \Pkint{k}(K)$ and $u^\perp \in \Pkperp{k}(K)$,
  whence $\trc(u) = \trc(u^\perp)$.
\end{proof}

\begin{definition}[Local and global test spaces]
  \label{def:testspace}
  Let $\Sc$ denote the set of sign vectors
  $s = (s_0, \dots, s_N) \in \{-1,0,+1\}^{N+1}$ with at least two
  nonzero entries.  For $k \ge 1$ and any sign vector $s \in \Sc$, set
  \begin{align}      \label{eq:VK}
    V_s(K) & := \bdr\, \Pkperp{k}(K) \,+\, \bint\, P_{k-2}(K)
           \;\subseteq\; P_{k+N}(K),
  \end{align}
  with $\bdr$ built from $s$ as in \eqref{eq:bdrbubble}. We
  emphasize that $V_s(K)$ depends on $s$ (through the boundary
  bubble $\bdr$, and only through it). Define the advection-adapted
  local space by
  \begin{equation}
    \label{eq:VKmethod}
    V(K) := V_{s(\beta_K)}(K),
  \end{equation}
  whose sign vector $s(\beta_K)$, defined in \eqref{eq:signs}, has
  at least two nonzero entries by
  Lemma~\ref{lem:two-noncharacteristic}.
  The local
  spaces \eqref{eq:VKmethod} are assembled into the Cartesian product
  \begin{equation}
    \label{eq:Vh}
    V_h := \bigtimes_{K \in \oh} V(K),
  \end{equation}
  identifiable as a subspace of $\W{\rho}{\oh}$ for every $1 \le \rho \le \infty$. Choose  $Y_h$ in~\eqref{eq:method} to be any computationally convenient space satisfying
  \begin{equation}
  \label{eq:Yh-condition}
  V_h \,\subseteq\, Y_h \,\subset\, \W{q}{\oh},
  \qquad \dim Y_h < \infty .
\end{equation}
(Since $V(K) \subset P_{k+N}(K)$, a simple choice is  $Y_h = P_{r}(\oh)$ for an $r \ge k+N$.)
\end{definition}

Our next assumption excludes characteristic facets, uniformly over
the mesh family. \emph{It is needed only for degrees $k \ge 2$.} All
our results in the lowest-order case $k = 1$ hold without it.

\begin{assumption}[Uniformly non-characteristic facets
  when $k \ge 2$]
  \label{asm:theta}
  There is a $\theta_0 \in (0, 1]$ such that for every $K \in \oh$
  and every facet $F_i$ of $K$,
  \begin{equation}
    \label{eq:noncharacteristic}
    | \beta_K \cdot n_i | \,\ge\, \theta_0\, |\beta_K| .
  \end{equation}
\end{assumption}

Assumption~\ref{asm:theta} forces the signs $s_i$ in \eqref{eq:signs} to lie in
$ \{-1,+1\}$. It is
automatically satisfied, e.g., when $\beta$ is a fixed constant
vector on $\Om$ and the mesh family avoids facets parallel to it. Note also
that on any single fixed mesh with no characteristic facets, some
$\theta_0 > 0$ exists trivially. But the content of the assumption is
its uniformity over a family of meshes.
All invocations of Assumption~\ref{asm:theta} in this section use only its implication that
$s_i \in \{-1,+1\}$, but the uniformity will be needed later.
Throughout this paper,
whenever $k \ge 2$ we require Assumption~\ref{asm:theta}, while for
$k = 1$ only Assumption~\ref{asm:beta} is required and
{\em characteristic facets are allowed.}

\begin{lemma}
  \label{lem:0-nonchar-facets}
  Let $\beta \in \R^N \setminus \{0\}$ and let $s = s(\beta)$
  determine the characteristic facets of $K$.
  Suppose that either Assumption~\ref{asm:theta} holds or $k =
  1$. Then any $u \in \Pkperp{k}(K)$ that vanishes on all
  non-characteristic facets must vanish everywhere.
\end{lemma}
\begin{proof}
  There are two cases:
  \begin{enumerate}[label=(\roman*)]
  \item When Assumption~\ref{asm:theta} holds,
    no facet is characteristic, and this gives
    $u|_{\partial K} = 0$, i.e.\
    $u \in \Pkint{k}(K) \cap \Pkperp{k}(K) = \{0\}$.
  \item If $k = 1$,  then
    $u$ is affine and vanishes on at least two facets
    by Lemma~\ref{lem:two-noncharacteristic}, so $u$ must vanish
    (since an affine function
    vanishing on a facet $F_i$ is a multiple of $\lambda_i$).
  \end{enumerate}
  In either case, $u\equiv 0$.
\end{proof}

\begin{lemma}[Structure of $V(K)$]
  \label{lem:structure}
  Suppose Assumption~\ref{asm:beta} holds, and that either
  Assumption~\ref{asm:theta} holds or $k = 1$.  Then the sum in
  \eqref{eq:VK} with $s = s(\beta_K)$ is direct and
  \begin{equation*}
    \dim V(K) = \dim \trc(\Pk(K)) + \dim P_{k-2}(K).
  \end{equation*}
\end{lemma}

\begin{proof}
  Since $\bdr$ and $\bint$ are nontrivial polynomials (see
  \eqref{eq:bdr-nonzero}), the multiplication maps $u \mapsto \bdr u$
  and $z \mapsto \bint z$ are injective. Hence
  $\dim(b P_{k-2}(K)) = \dim P_{k-2}(K)$, and using also
  Lemma~\ref{lem:extension},
  $\dim \bdr \Pkperp{k}(K) = \dim \Pkperp{k}(K) = \dim \trc(\Pk(K))$.

  It only remains to show the intersection of the two summands is
  trivial. Suppose $\bdr u = \bint z$ with $u \in \Pkperp{k}(K)$,
  $z \in P_{k-2}(K)$. Restricting to $\partial K$ and using
  $\bint|_{\partial K} = 0$ gives $\bdr u = 0$ on $\partial K$.  Since
  $\bdr|_{F_i} = s_i b_i \neq 0$ on $\intr{F}_i$ for every
  non-characteristic $F_i$, the trace of $u$ vanishes on all
  non-characteristic facets. By Lemma~\ref{lem:0-nonchar-facets},
  $u = 0$, so $\bint z = 0$, which implies $z = 0$.
\end{proof}

Lemma~\ref{lem:structure} shows that $V(K)$ is \emph{minimal} in the
sense that $\dim V(K)$ equals the number of Fortin
conditions~\eqref{eq:fortin-conditions} of the next section, with
\eqref{eq:fortin-conditions} being the smallest set of conditions that
forces the Fortin property, since the trial degrees of freedom active
on $K$ enter \eqref{eq:bform} only through $\trc(\Pk(K))$ and through
$\beta\cdot\nabla P_{k-1}(K) = P_{k-2}(K)$.

\section{The Fortin operator}
\label{sec:fortin}

This section establishes the unisolvency of equations defining our first Fortin operator, deferring all quantitative
bounds to Section~\ref{sec:continuity}, which draws on the scaling
estimates of Section~\ref{sec:scaling}.

\begin{lemma}[Definition of $\varPi_K$]
  \label{lem:fortin-def}
  Suppose Assumption~\ref{asm:beta} holds, and that either
  Assumption~\ref{asm:theta} holds or $k = 1$. For every
  $v \in \W{q}{K}$ there is a unique
  $\varPi_K v \in V(K)$ such that
  \begin{subequations}
    \label{eq:fortin-conditions}
    \begin{alignat}{2}
      \ap{D_K w, \varPi_K v - v} &= 0
      & \qquad & \text{ for all } \, w \in \Pkperp{k}(K),
      \label{eq:F1}\\
      \ip{z}{\varPi_K v - v}{K} &= 0
      & & \text{ for all } \, z \in P_{k-2}(K).
      \label{eq:F2}
    \end{alignat}
  \end{subequations}
  Moreover,
  \eqref{eq:F1} continues to hold for all $w \in \Pk(K)$.
  (For $k = 1$, condition \eqref{eq:F2} is vacuous.)
\end{lemma}
\begin{proof}
  The terms $\ap{D_K w,  v}$ and $\ip{z}{v}{K} $ are well defined
  when $v \in \W{q}{K}$ (and when $z, w$ are polynomials) due to \eqref{eq:DK}.
  By Lemmas~\ref{lem:extension} and \ref{lem:structure}, the number of
  (independent) conditions in \eqref{eq:fortin-conditions} equals
  $\dim \Pkperp{k}(K) + \dim P_{k-2}(K) = \dim V(K)$, so
  \eqref{eq:fortin-conditions} is a square linear system for
  $\varPi_K v$ and it suffices to prove uniqueness. Accordingly, we
  need to prove that any $\varphi = \varPi_K v \in V(K)$ satisfying
  \eqref{eq:fortin-conditions} with $v = 0$ must vanish.

  By Lemma~\ref{lem:structure} we may decompose
  \begin{equation*}
    \varphi = \bdr\, \vartheta + \bint\, \psi,
    \qquad \vartheta \in \Pkperp{k}(K), \quad \psi \in P_{k-2}(K),
  \end{equation*}
  uniquely. By~\eqref{eq:F1}, \eqref{eq:DKbdry} and \eqref{eq:bdr-pos},
  since $\bint$ vanishes on $\partial K$,
  \begin{equation*}
    0 = \ap{D_K w, \varphi}
    = \int_{\partial K} (\beta\cdot n)\, w \,\bdr\, \vartheta \,ds
    = \sum_{i=0}^N \int_{F_i} |\beta \cdot n|\; b_i\, w\, \vartheta \,ds,
  \end{equation*}
  for any $w \in \Pkperp{k}(K)$. Choose $w = \vartheta$. Since every
  summand is then nonnegative and $b_i > 0$ on $\intr{F}_i$, the trace
  of $\vartheta$ must vanish on every non-characteristic facet. By Lemma~\ref{lem:0-nonchar-facets},
  $\vartheta = 0$.  Thus $\varphi = \bint \psi$, and choosing the test function
  $z$ in \eqref{eq:F2} to be $\psi$ gives $( \bint \psi, \psi)_K = 0$, which
  implies that $\psi = 0$. Thus $\varphi = 0$.

  The final claim that \eqref{eq:F1} continues to hold for all
  $w \in \Pk(K)$ is seen by splitting $w \in \Pk(K)$ into
  $w = \mathring w + w^\perp$ with $\mathring w \in \Pkint{k}(K)$,
  $w^\perp \in \Pkperp{k}(K)$, and observing that
  \begin{equation*}
    \ap{D_K \mathring w, \varPi_K v - v} = 0,
  \end{equation*}
  because $\mathring w$ vanishes on $\partial K$.
\end{proof}

\begin{remark}[Failure at an exactly characteristic facet when $k=2$]
  \label{rem:fortin-fails}
  Suppose $N = 2$, $k = 2$, and suppose $\beta_K \cdot n_0 = 0$ for
  the facet $F_0$ of a triangle $K$ (so $s_0 = 0$), while $F_1, F_2$
  are non-characteristic with nonzero signs $s_1, s_2$.  Then clearly
  Assumption~\ref{asm:theta} does not hold.  We show that the result
  of Lemma~\ref{lem:fortin-def} cannot hold.  Note that
  $\bdr = s_1 b_1 + s_2 b_2 = \lambda_0 (s_1 \lambda_2 + s_2
  \lambda_1)$, since the $b_0 = \lambda_1 \lambda_2$ term is absent.
  Let $u := \lambda_1 \lambda_2$. Clearly $u$ is in
  $\Pkperp{2}(K) = P_2(K)$ as $k = 2 = N$ and
  $    \bdr\, u
    \,=\, \lambda_0 \lambda_1 \lambda_2\, (s_1 \lambda_2 + s_2 \lambda_1)
    \,=\, \bint\, (s_1 \lambda_2 + s_2 \lambda_1).
 $
  Because $\bint$ vanishes identically on $\partial K$, so does
  $\bdr\, u$. Hence $\ap{D_K w, \bdr\, u} = 0$ for every
  $w \in \Pkperp{2}(K)$, i.e., \eqref{eq:F1} holds trivially.
  Set
  $z_0 = (\bdr\, u, 1)_K / (\bint, 1)_K$. Then
  $
  \varphi \,:=\, \bdr\, u - \bint\, z_0
  = \bint\, (s_1 \lambda_2 + s_2 \lambda_1 - z_0)
  $
  has zero mean and, since $s_1 \lambda_2 + s_2 \lambda_1$ is linear
  but not constant, $\varphi$
  is a nonzero function in $V(K)$. Since $\varphi$
  vanishes on $\partial K$, \eqref{eq:F1} holds and since $\varphi$
  has zero mean, \eqref{eq:F2} holds too, showing that $\varPi_K$
  cannot be uniquely defined by~\eqref{eq:fortin-conditions}.
\end{remark}

\begin{definition}
  \label{def:fortin-global}
  Define $\varPi_h : \W{q}{\oh} \to V_h$ element-wise by
  $(\varPi_h v)|_K := \varPi_K (v|_K)$.
\end{definition}

\begin{lemma}[Fortin property]
  \label{lem:fortin-property}
  Suppose Assumption~\ref{asm:beta} holds, and that either
  Assumption~\ref{asm:theta} holds or $k = 1$. For all
  $v \in \W{q}{\oh}$
  and all $(w, \mu) \in P_{k-1}(\oh) \times \Xhk$,
  \begin{equation*}
    \bform\big( (w, \mu),\, \varPi_h v - v \big) = 0 .
  \end{equation*}
\end{lemma}

\begin{proof}
  By \eqref{eq:trial}, $\mu=\Dh z_h$ for some $z_h\in S_h^k$.
  Let $\delta := \varPi_h v - v$ and fix $K \in \oh$. Since
  $w|_K \in P_{k-1}(K) \subset \Pk(K)$, the definition \eqref{eq:DK}
  gives
  \begin{equation*}
    -\ip{w}{\beta\cdot\nabla \delta}{K}
    = \ip{\beta\cdot\nabla w}{\delta}{K}
    - \ap{D_K w, \delta} .
  \end{equation*}
  Here $\beta\cdot\nabla w \in P_{k-2}(K)$ because $\beta_K$ is
  constant, so the first term vanishes by \eqref{eq:F2}. Combining
  this identity with \eqref{eq:trfunctional}, we obtain
  \begin{equation*}
    -\ip{w}{\beta\cdot\nabla \delta}{K}
    +
    \ap{D_K(z_h|_K),\delta}
    =
    \ap{D_K(z_h|_K-w|_K),\delta}.
  \end{equation*}
  The last term vanishes by Lemma~\ref{lem:fortin-def}, because
  $z_h|_K-w|_K\in\Pk(K)$. Summing over $K$ completes the proof.
\end{proof}

\section{Scaling estimates}
\label{sec:scaling}

This section develops the element-wise estimates needed to bound
$\varPi_K$. The strategy is to prove equivalences on the reference
simplex $\Khat$ by homogeneity and
compactness, treating the advection vector as a variable, and then to
transfer them to a physical element $K$ by the affine map
\eqref{eq:affinemap}, tracking how every quantity scales.

\subsection{Pullback conventions}
For $v$ defined on $K$ write $\hat v := v \circ \Phi_K$. The chain
rule gives $\hat\nabla \hat v = M_K^t (\nabla v)\circ \Phi_K$, whence
for any constant vector $\beta \in \R^N$,
\begin{equation}
  \label{eq:pullback-grad}
  (\beta \cdot \nabla v) \circ \Phi_K
  = \hat\beta \cdot \hat\nabla \hat v,
  \qquad
  \hat\beta := M_K^{-1} \beta.
\end{equation}
Since $\Phi_K$ maps vertices to vertices,
$\lambda_i \circ \Phi_K = \hat \lambda_i$, so the bubbles satisfy
$b_i \circ \Phi_K = \hat b_i$ and $\bint \circ \Phi_K = \hat \bint$,
where $\hat b_i$ and $\hat \bint$ are defined by
\eqref{eq:bubbles} using $\hat\lambda_i$ in place of $\lambda_i$.
Moreover, because the
Jacobian of $\Phi_K$ is constant, pullback preserves $L_2$
orthogonality, so
\begin{equation}
  \label{eq:pullback-spaces}
  \{ \hat u = u \circ \Phi_K : u \in \Pkint{k}(K)\} = \Pkint{k}(\Khat),
  \qquad
  \{ \hat u = u \circ \Phi_K:  u \in \Pkperp{k}(K)\} = \Pkperp{k}(\Khat).
\end{equation}

\begin{lemma}[Sign pattern and non-characteristicity are affine
  invariants]
  \label{lem:signs}
  Let $\hat n_i$ and $n_i$ denote the outward unit normals of
  $\Khat$ and $K$ on corresponding facets, and let
  $\beta \in \R^N \setminus \{0\}$ and
  $\hat\beta = M_K^{-1}\beta$. Then
  the sign vector of~\eqref{eq:signs} is mapped exactly,
  \begin{equation*}
    \sgn ( \hat\beta \cdot \hat n_i ) = \sgn (\beta \cdot n_i),
    \qquad \text{i.e.,} \qquad
    s(\hat\beta) = s(\beta),
  \end{equation*}
    and if
  $|\beta \cdot n_i| \ge \theta_0 |\beta|$ for all $i$, then
  \begin{equation*}
    | \hat\beta \cdot \hat n_i |
    \,\ge\, \frac{\theta_0}{\kappa_K}\, |\hat\beta|
  \end{equation*}
  for all $i$, where $\kappa_K$ is as in \eqref{eq:shapereg}.
\end{lemma}

\begin{proof}
  Since $\lambda_i$ is affine, vanishes on $F_i$, and equals $1$ at
  $a_i$, its (constant) gradient $\nabla\lambda_i$ points from $F_i$
  toward $a_i$, so the outward unit normal is
  $n_i = -\nabla\lambda_i / |\nabla\lambda_i|$. Likewise
  $\hat n_i = -\hat\nabla\hat\lambda_i / |\hat\nabla\hat\lambda_i|$ on
  $\Khat$. Applying \eqref{eq:pullback-grad} to the affine function
  $v = \lambda_i$ (whose gradient is constant, so composing with
  $\Phi_K$ has no effect) gives the equality of constants
  $\beta \cdot \nabla\lambda_i = \hat\beta \cdot \hat\nabla\hat\lambda_i$.
  Hence
  \begin{equation}
    \label{eq:normal-transform}
    \hat\beta \cdot \hat n_i
    \,=\, - \frac{\hat\beta \cdot \hat\nabla \hat\lambda_i}
    {|\hat\nabla\hat\lambda_i|}
    \,=\, - \frac{\beta \cdot \nabla\lambda_i}{|\hat\nabla\hat\lambda_i|}
    \,=\, \frac{|\nabla\lambda_i|}{|\hat\nabla\hat\lambda_i|}\,
    (\beta \cdot n_i) ,
  \end{equation}
  a positive multiple of $\beta \cdot n_i$, which  proves the sign
  assertion.
  For the second claim, the chain rule
  $\hat\nabla\hat\lambda_i = M_K^t \nabla\lambda_i$ gives
  $|\hat\nabla\hat\lambda_i| \le |M_K| \, |\nabla\lambda_i|$, so
  \eqref{eq:normal-transform} yields
  \begin{equation*}
    |\hat\beta \cdot \hat n_i|
    \ge \frac{|\beta \cdot n_i|}{|M_K|}
    \ge \frac{\theta_0 |\beta|}{|M_K|}
    \ge \frac{\theta_0\, |\hat\beta|}{|M_K| |M_K^{-1}|}
    = \frac{\theta_0}{\kappa_K} |\hat\beta|,
  \end{equation*}
  using $|\hat\beta| \le |M_K^{-1}| |\beta|$.
\end{proof}

\subsection{Parametrization of bubbles}

Throughout this subsection, let $s$ be a sign vector in $\Sc$
(see Definition~\ref{def:testspace}), and let $\bdr$ and $\hbdr$ be
as in \eqref{eq:bdrbubble} on $K$ and $\Khat$, respectively. Fix once and for all a basis
$\hat\psi_1, \dots, \hat\psi_m$ of $\Pkperp{k}(\Khat)$, where
\begin{equation*}
  m := \dim \Pkperp{k}(\Khat) = \dim \trc(\Pk(\Khat))
  = \binom{N+k}{N} - \binom{k-1}{N},
\end{equation*}
and let $\psi_i := \hat\psi_i \circ \Phi_K^{-1} \in \Pkperp{k}(K)$
(see \eqref{eq:pullback-spaces}). For $c \in \R^m$ define
\begin{equation}
  \label{eq:param}
  w_K(c) := \sum_{i=1}^m c_i \psi_i \in \Pkperp{k}(K),
  \qquad
  \varphi_K(c,s) := \bdr \, w_K(c) \in \bdr \Pkperp{k}(K)
  \subset V_s(K),
\end{equation}
and analogously $w_{\Khat}(c)$, $\varphi_{\Khat}(c,s)$ on $\Khat$, so
that $w_K(c) \circ \Phi_K = w_{\Khat}(c)$ and
$\varphi_K(c,s) \circ \Phi_K = \varphi_{\Khat}(c,s)$. For
$1 \le p \le \infty$ and a constant vector $\beta \ne 0$ with
$s(\beta) = s$ (see \eqref{eq:signs}), define
\begin{subequations}
  \label{eq:Sfd}
  \begin{align}
    S_{p,K}(\beta, c, s) &:=
    \nrmK{\beta \cdot \nabla \varphi_K(c, s)}{p}{K},
    &
    \tilde S_{p,K}(\beta, c) &:=
    \nrmK{\beta \cdot \nabla w_K(c)}{p}{K},
    \label{eq:Sdef}
    \\
    f_{p,K}(c, s) &:= \nrmK{\varphi_K(c, s)}{p}{K},
    &
    \tilde f_{p,K}(c) &:= \nrmK{w_K(c)}{p}{K},
    \label{eq:fdef}
    \\
    d_K(\beta, c) &:=
    \rlap{$\displaystyle\sum_{i=0}^N \int_{F_i} |\beta\cdot n| \;
      b_i \, w_K(c)^2 \, ds . $}
    \label{eq:ddef}
  \end{align}
\end{subequations}
By \eqref{eq:bdr-pos} and \eqref{eq:DKbdry},
using $\varphi_K(c, s)$ with  $s = s(\beta)$,
\begin{equation}
  \label{eq:d-identity}
  d_K(\beta, c)
  = \int_{\partial K} (\beta \cdot n)\, w_K(c)\, \varphi_K(c, s(\beta))\, ds
  = \ap{D_K w_K(c), \varphi_K(c, s(\beta))} .
\end{equation}

\begin{lemma}[Reference element estimates]
  \label{lem:equiv}
  Let $1 \le p \le \infty$. There exist positive constants
  $C_0, C_1, C_2$ depending only on $N, k, p$, and a positive constant
  $C_0^\theta$ depending on $N, k, p$ and additionally on $\theta$, all
  independent of $s$, such that for all $c \in \R^m$ and all
  $\beta \in \R^N\setminus\{0\}$ with $s(\beta) = s$, the following
  statements hold:
  \begin{enumerate}
  \item On the reference element $\hat K$,
    \begin{subequations}
      \label{eq:C1C2}
      \begin{gather}
        S_{p,\Khat}(\beta, c, s) + \tilde S_{p,\Khat}(\beta, c)
        \le C_2\, |\beta|\, |c| ,
        \label{eq:S-upper}
        \\
        C_1 |c| \,\le\, f_{p,\Khat}(c, s) \le C_2 |c|,
        \qquad
        C_1 |c| \,\le\, \tilde f_{p,\Khat}(c) \le C_2 |c| ,
        \label{eq:f-equiv}
        \\
        d_{\Khat}(\beta, c) \le C_2\, |\beta|\, |c|^2.
        \label{eq:d-upper}
      \end{gather}
    \end{subequations}
  \item If $k=1$, then
    \begin{equation}
      \label{eq:d-lower}
      C_0\, |\beta| \, |c|^2 \,\le\, d_{\Khat}(\beta, c).
    \end{equation}
  \item If there is a $\theta \in (0,1]$ such that
    $|\beta \cdot \hat n_i| \ge \theta |\beta|$ for all $i$, then
    \begin{equation}
      \label{eq:d-lower-theta}
      C_0^\theta\, |\beta| \, |c|^2 \,\le\, d_{\Khat}(\beta, c).
    \end{equation}
  \end{enumerate}
\end{lemma}
\begin{proof}
  It suffices to prove~\eqref{eq:C1C2} with the constants $C_1(s)$ and
  $C_2(s)$ depending on~$s$ for the $s$-dependent quantities. Since $s$
  ranges over the finite subset $\Sc$ (of Definition~\ref{def:testspace})
  the same inequalities then continue to
  hold with $C_2 = \max_{s \in \Sc}C_2(s)$ in place of $C_2(s)$ and
  $C_1 = \min_{s \in \Sc} C_1(s)$, respectively. (The constants
  in~\eqref{eq:d-lower} and~\eqref{eq:d-lower-theta} are obviously
  $s$-independent since those
  inequalities govern $s$-independent quantities.)
  
  All quantities are positively homogeneous in $|\beta|$ and $|c|$:
  indeed $c \mapsto w_{\Khat}(c)$ and $c \mapsto \varphi_{\Khat}(c,s)$
  are linear, and
  \eqref{eq:ddef} is homogeneous of degree~$1$ in $|\beta|$ (as seen
  by writing
  $|\beta\cdot \hat n| = |\beta| \, |(\beta/|\beta|) \cdot \hat n|$)
  and of degree~$2$ in $c$. Hence it
  suffices to prove all bounds for $|\beta| = |c| = 1$.
  The upper bounds follow because each quantity is a continuous
  function of $(\beta, c)$ on the compact set
  $\{|\beta| = 1\} \times \{|c| = 1\}$.

  For the lower bounds, start with~\eqref{eq:f-equiv}. Observe that
  $f_{p,\Khat}$ and $\tilde f_{p,\Khat}$ are norms of linear images of
  $c$, so $\tilde f_{p,\Khat}(c)$ vanishes only if $w_{\Khat}(c) = 0$,
  i.e.\ only if $c = 0$, since the $\hat\psi_i$ form a
  basis. Similarly, $f_{p,\Khat}(c, s)$ vanishes only if
  $\hbdr\, w_{\Khat}(c) = 0$, which forces $w_{\Khat}(c) = 0$ and
  $c=0$. Both $f_{p,\Khat}(c, s)$ and $\tilde f_{p,\Khat}(c)$ are therefore
  norms on $\R^m$, equivalent to $|c|$, with constants depending only
  on $(N, k, p)$ and the fixed basis.

  For \eqref{eq:d-lower-theta}, consider the set
  \begin{equation*}
    A_\theta := \big\{ (\beta, c) : |\beta| = 1, \ |c| = 1, \
    |\beta \cdot \hat n_i| \ge \theta \ \text{ for } i = 0,\dots,N
    \big\},
  \end{equation*}
  which is compact. The
  function $d_{\Khat}$ is obviously continuous on $A_\theta$. It is also
  positive there. Indeed, each summand of \eqref{eq:ddef}
  defining $d_{\Khat}$ is nonnegative, so
  if all of them vanish, then, since
  $|\beta \cdot \hat n| \ge \theta > 0$ on each facet and
  $\hat b_i > 0$ on $\intr{\hat F}_i$, the trace of
  $w_{\Khat}(c)$
  vanishes on every facet, thus making
  $w_{\Khat}(c) \in \Pkint{k}(\Khat) \cap \Pkperp{k}(\Khat) = \{0\}$,
  contradicting $|c| = 1$. The minimum of $d_{\Khat}$ over $A_\theta$
  is attained and positive; call it $C_0^\theta$.

  For~\eqref{eq:d-lower}, we apply the same argument noting that when
  $k = 1$ we may use the larger compact set
  $A := \{ (\beta, c) : |\beta| = 1, \ |c| = 1 \}$.  If
  $d_{\Khat}(\beta,c) = 0$, then the trace of the affine function
  $w_{\Khat}(c)$ vanishes on every facet with
  $\beta \cdot \hat n_i \neq 0$ (of which there are at least two, by
  Lemma~\ref{lem:two-noncharacteristic}) and an affine function
  vanishing on two distinct facets is identically zero, contradicting
  $|c| = 1$. Hence for $k = 1$ the minimum of $d_{\Khat}$ over all of
  $A$, named $C_0$, is positive (and $C_0$ is $\theta$-free).
\end{proof}

\begin{remark}
  No lower bound for $S_{p,\Khat}$ of the form
  $C|\beta||c| \le  S_{p,\Khat}(\beta,c,s)$ is claimed in
  Lemma~\ref{lem:equiv}. For $k \ge 2$ such a bound is false in
  general, since $\beta\cdot\nabla\varphi_{\Khat}(c,s)$ can vanish for
  $c \neq 0$ when $\varphi_{\Khat}(c,s)$ is constant along
  streamlines. It is
  also never needed in our analysis.
\end{remark}

\begin{lemma}[Scaling identities]
  \label{lem:scaling}
  Let $1 \le p \le \infty$ (with the convention $1/\infty = 0$),
  $c \in \R^m$, and let $\beta \neq 0$ satisfy $s(\beta) = s$ on $K$.
  Then, with $\hat\beta = M_K^{-1}\beta$,
  \begin{subequations}
    \begin{align}
      S_{p,K}(\beta, c, s)
      &= \Big( \tfrac{|K|}{|\Khat|} \Big)^{1/p}
      S_{p,\Khat}(\hat\beta, c, s),
      &
      \tilde S_{p,K}(\beta, c)
      &= \Big( \tfrac{|K|}{|\Khat|} \Big)^{1/p}
      \tilde S_{p,\Khat}(\hat\beta, c),
      \label{eq:S-scaling}
      \\
      f_{p,K}(c,s)
      &= \Big( \tfrac{|K|}{|\Khat|} \Big)^{1/p}
      f_{p,\Khat}(c,s),
      &
      \tilde f_{p,K}(c)
      &= \Big( \tfrac{|K|}{|\Khat|} \Big)^{1/p}
      \tilde f_{p,\Khat}(c),
      \label{eq:f-scaling}
    \end{align}
    \begin{equation}
      d_K(\beta, c)
      = \frac{|K|}{|\Khat|} \; d_{\Khat}(\hat\beta, c) .
      \label{eq:d-scaling}
    \end{equation}
  \end{subequations}
\end{lemma}
\begin{proof}
  Since $\varphi_K(c,s) \circ \Phi_K = \varphi_{\Khat}(c,s)$ and
  $w_K(c) \circ \Phi_K = w_{\Khat}(c)$, the identities
  \eqref{eq:S-scaling}--\eqref{eq:f-scaling} follow from
  \eqref{eq:pullback-grad} and the change of variables formula (for
  $p = \infty$, suprema are invariant under composition with
  $\Phi_K$). For \eqref{eq:d-scaling} we use the representation
  \eqref{eq:d-identity} with $s = s(\beta)$ and integrate by parts
  using~\eqref{eq:DK}:
  \begin{align*}
    d_K(\beta, c)
    &= \ip{\beta\cdot\nabla w_K(c)}{\varphi_K(c,s)}{K}
    + \ip{w_K(c)}{\beta\cdot\nabla \varphi_K(c,s)}{K}
    \\
    &= \frac{|K|}{|\Khat|} \Big[
    \ip{\hat\beta\cdot\hat\nabla w_{\Khat}(c)}
       {\varphi_{\Khat}(c,s)}{\Khat}
    + \ip{w_{\Khat}(c)}
         {\hat\beta\cdot\hat\nabla \varphi_{\Khat}(c,s)}{\Khat}
    \Big]
    = \frac{|K|}{|\Khat|}\, d_{\Khat}(\hat\beta, c),
  \end{align*}
  where the last equality is \eqref{eq:d-identity} on $\Khat$, valid
  because $s(\hat\beta) = s(\beta) = s$ on $\Khat$ by
  Lemma~\ref{lem:signs}.
\end{proof}

\begin{lemma}[Interior bubble estimate]
  \label{lem:bubble}
  Let $j\ge0$, let $1\le r<\infty$, and let $r'$ be the conjugate
  exponent. There is a $C=C(N,j,r)$ such that for every simplex $K$
  and every $z\in P_j(K)$,
  \begin{equation*}
    \nrmK{\bint z}{r}{K}\;\nrmK{z}{r'}{K}
    \,\le\, C \ip{z}{\bint z}{K} .
  \end{equation*}
\end{lemma}

\begin{proof}
  On $\Khat$, the three maps
  $z \mapsto \nrmK{\hat \bint z}{r}{\Khat}$,
  $z \mapsto \nrmK{z}{r'}{\Khat}$ and
  $z \mapsto \ip{z}{\hat \bint z}{\Khat}^{1/2}$ are norms on the
  finite dimensional space $P_j(\Khat)$ (the first because
  $\hat\bint \neq 0$ and polynomials form an integral domain, the
  third because $\hat\bint > 0$ in the interior of $\Khat$), hence
  pairwise equivalent with constants depending only on $(N,j,r)$.
  This proves the claim on $\Khat$. Mapping to $K$, the left-hand
  side scales by
  $(|K|/|\Khat|)^{1/r} (|K|/|\Khat|)^{1/r'} = |K|/|\Khat|$ and the
  right-hand side by $|K|/|\Khat|$ as well, so the inequality
  transfers with the same constant.
\end{proof}

From now on, $A\lsim B$ means that $A\le CB$ with
a positive constant $C$ depending on $N,k,p,\kappa_0$, and generally
on $\theta_0$, but independent of $\theta_0$ when $k=1$.

\begin{lemma}[Element estimates]
  \label{lem:element}
  Suppose Assumption~\ref{asm:beta} and \eqref{eq:shapereg} hold,
  and that either Assumption~\ref{asm:theta} holds or $k = 1$. Let
  $1 \le p < \infty$, let $q$ be as in \eqref{eq:conjugate}, let
  $K \in \oh$, $\beta = \beta_K$, and let $s := s(\beta)$ be its sign
  vector \eqref{eq:signs}. Then, for all
  $c \in \R^m$,
  \begin{subequations}
    \begin{align}
      f_{p,K}(c,s)^2
      &\lsim  |K|^{\frac1p - \frac1q} \;
      \frac{|M_K|}{|\beta|} \; d_K(\beta, c),
      \label{eq:f2-dbound}
      \\
      \tilde f_{q,K}(c)
      &\lsim |K|^{\frac1q - \frac1p}\; f_{p,K}(c,s),
      \label{eq:ftilde-bound}
      \\
      \tilde S_{q,K}(\beta, c)
      &\lsim |K|^{\frac1q - \frac1p}\; |M_K^{-1}|\,|\beta| \;
      f_{p,K}(c,s).
      \label{eq:Stilde-bound}
    \end{align}
  \end{subequations}
\end{lemma}

\begin{proof}
  Throughout, $\hat\beta = M_K^{-1}\beta$, which by
  Lemma~\ref{lem:signs} satisfies $s(\hat\beta) = s(\beta) = s$ on
  $\Khat$ and, when $k \ge 2$, satisfies
  $|\hat\beta \cdot \hat n_i| \ge (\theta_0/\kappa_0) |\hat\beta|$ for
  all $i$; thus Lemma~\ref{lem:equiv} applies on $\Khat$, giving
  \eqref{eq:d-lower-theta} with $\theta = \theta_0/\kappa_0$ when
  $k \ge 2$, and \eqref{eq:d-lower} (with no $\theta$-restriction)
  when $k = 1$. Also note $1/|\hat\beta| \le |M_K|/|\beta|$, since
  $|\beta| = |M_K\hat\beta| \le |M_K||\hat\beta|$. Factors of
  $|\Khat|$ are absorbed into constants.

  Proof of \eqref{eq:f2-dbound}:
  \begin{align*}
    f_{p,K}(c,s)^2
    &= \Big( \tfrac{|K|}{|\Khat|} \Big)^{2/p}
    f_{p,\Khat}(c,s)^2
    && \text{by \eqref{eq:f-scaling}}
    \\
    &\lsim \Big( \tfrac{|K|}{|\Khat|} \Big)^{2/p} |c|^2
    \lsim \Big( \tfrac{|K|}{|\Khat|} \Big)^{2/p}
    \frac{ d_{\Khat}(\hat\beta, c)}{|\hat\beta|}
    && \text{by \eqref{eq:f-equiv}, \eqref{eq:d-lower}/\eqref{eq:d-lower-theta}}
    \\
    &= \Big( \tfrac{|K|}{|\Khat|} \Big)^{2/p - 1}
    \frac{ d_{K}(\beta, c)}{|\hat\beta|}
    \;\le\; \Big( \tfrac{|K|}{|\Khat|} \Big)^{2/p-1}
    \frac{|M_K|}{|\beta|}\; d_K(\beta, c)
    && \text{by \eqref{eq:d-scaling}},
  \end{align*}
  and $2/p - 1 = 1/p - 1/q$.

  Proof of \eqref{eq:ftilde-bound}:
  \begin{equation*}
    \tilde f_{q,K}(c)
    = \Big( \tfrac{|K|}{|\Khat|} \Big)^{1/q}
    \tilde f_{q,\Khat}(c)
    \,\lsim\, \Big( \tfrac{|K|}{|\Khat|} \Big)^{1/q} |c|
    \,\lsim\, \Big( \tfrac{|K|}{|\Khat|} \Big)^{1/q}
    f_{p,\Khat}(c,s)
    = \Big( \tfrac{|K|}{|\Khat|} \Big)^{1/q - 1/p}
    f_{p,K}(c,s),
  \end{equation*}
  by \eqref{eq:f-scaling} (twice, with exponents $q$ and $p$) and
  \eqref{eq:f-equiv} (upper bound at exponent $q$, lower bound at
  exponent $p$).

  Proof of \eqref{eq:Stilde-bound}: Similarly,
  \begin{equation*}
    \tilde S_{q,K}(\beta, c)
    = \Big( \tfrac{|K|}{|\Khat|} \Big)^{1/q}
    \tilde S_{q,\Khat}(\hat\beta, c)
    \,\lsim\, \Big( \tfrac{|K|}{|\Khat|} \Big)^{1/q}
    |\hat\beta|\, |c|
    \,\lsim\, \Big( \tfrac{|K|}{|\Khat|} \Big)^{1/q - 1/p}
    |M_K^{-1}|\, |\beta| \, f_{p,K}(c,s),
  \end{equation*}
  by \eqref{eq:S-scaling}, \eqref{eq:S-upper},
  $|\hat\beta| \le |M_K^{-1}||\beta|$, and \eqref{eq:f-equiv} with
  \eqref{eq:f-scaling} as before.
\end{proof}

\section{Continuity of Fortin and augmented Fortin operators}
\label{sec:continuity}

\subsection{Bounding  the Fortin operator}
\label{sec:boundedness}

For $1 < p < \infty$, $K \in \oh$ and $v \in \W{q}{K}$, define
\begin{equation}
  \label{eq:norm}
  \tnrm{v}{q, K}
  := \nrmK{v}{q}{K}
  + \frac{h_K}{|\beta_K|}\,
  \nrmK{\beta\cdot\nabla v}{q}{K},
  \qquad
  \tnrm{v}{q,\oh}^q := \sum_{K \in \oh} \tnrm{v}{q,K}^q .
\end{equation}
On each fixed $K$ this norm generates the same topology as the
canonical norm $\nrmK{v}{q}{K} + \nrmK{\beta\cdot\nabla v}{q}{K}$ of
$\W{q}{K}$, but with equivalence constants depending on
$h_K/|\beta_K|$. The weighting in~\eqref{eq:norm} is chosen so that
the constant $C$ of Theorem~\ref{thm:main} below depends on neither
$h_K$ nor $|\beta_K|$, and is therefore uniform over the
shape-regular mesh family and invariant under rescaling of $\beta$.
In the proof below, we will also use the standard polynomial inverse
estimate: for any $j \ge 0$ and $1 \le r \le \infty$,
\begin{equation}
  \label{eq:inverse}
  \nrmK{\nabla g}{r}{K}  \lsim  |M_K^{-1}|\, \nrmK{g}{r}{K}
  \qquad \text{for all } g \in P_j(K).
\end{equation}

\begin{theorem}[Boundedness of $\varPi_h$]
  \label{thm:main}
  Suppose Assumption~\ref{asm:beta} and the shape regularity
  bound~\eqref{eq:shapereg} hold, and that either
  Assumption~\ref{asm:theta} holds or $k = 1$. Let
  $1 < p < \infty$. There is a constant
  $C = C(N, k, q, \theta_0, \kappa_0)$, independent of $\theta_0$ in
  the lowest-order case $k = 1$, such that for all
  $K \in \oh$ and $v \in \W{q}{K}$,
  \begin{equation}
    \label{eq:varPi-local-bound}
    \tnrm{\varPi_K v}{q,K} \,\le\, C\, \tnrm{v}{q,K} .
  \end{equation}
  Consequently
  $\tnrm{\varPi_h v}{q,\oh} \le C \tnrm{v}{q,\oh}$ for all
  $v \in \W{q}{\oh}$.
  Moreover, if $v$ is such that
  \begin{equation}
    \label{eq:annihilated}
    \ip{z}{v}{K} = 0 \quad \text{for all } z \in P_{k-2}(K),
    \qquad
    \ip{\beta\cdot\nabla w}{v}{K} = 0 \quad \text{for all } w \in
    \Pkperp{k}(K),
  \end{equation}
  then the sharper bounds
  \begin{equation}
    \label{eq:varPi-refined}
    \nrmK{\varPi_K v}{q}{K} \,\le\, C\, a_K\,
    \nrmK{\beta\cdot\nabla v}{q}{K},
    \qquad
    \nrmK{\beta\cdot\nabla \varPi_K v}{q}{K} \,\le\, C\,
    \nrmK{\beta\cdot\nabla v}{q}{K}
  \end{equation}
  hold, where $a_K := h_K/|\beta_K|$.
\end{theorem}

\begin{proof}
  Write $\beta = \beta_K$ and $s := s(\beta)$, fix $v \in \W{q}{K}$,
  and let
  $\varphi := \varPi_K v$. By Lemma~\ref{lem:structure} and the
  parametrization \eqref{eq:param}, there are unique $c \in \R^m$ and
  $z \in P_{k-2}(K)$ such that
  \begin{equation}
    \label{eq:Pi-v-split}
    \varphi = \varphi^{\perp} + \mathring\varphi,
    \qquad
    \varphi^\perp = \varphi_K(c,s) = \bdr\, w, \quad w := w_K(c),
    \qquad
    \mathring\varphi = \bint z .
  \end{equation}

  \emph{Step 1.} Let us  bound
  $\| \varphi^\perp\|_{L_q(K)} = f_{q,K}(c, s)$.
  Since $\mathring\varphi$ is a polynomial vanishing on $\partial K$,
  equation~\eqref{eq:DKbdry} gives
  $\ap{D_K w, \mathring\varphi} = 0$, so by
  \eqref{eq:d-identity} and the Fortin condition \eqref{eq:F1},
  \begin{equation*}
    d_K(\beta, c)
    = \ap{D_K w, \varphi^\perp}
    = \ap{D_K w, \varphi}
    = \ap{D_K w, v}
    = \ip{\beta\cdot\nabla w}{v}{K} + \ip{w}{\beta\cdot\nabla v}{K} .
  \end{equation*}
  H\"older's inequality with the conjugate exponents
  \eqref{eq:conjugate}, followed by \eqref{eq:Stilde-bound} and
  \eqref{eq:ftilde-bound} applied with $q$ in place of their generic
  exponent $p$, yields
  \begin{align*}
    d_K(\beta, c)
    &\le \tilde S_{p,K}(\beta,c)\, \nrmK{v}{q}{K}
    + \tilde f_{p,K}(c)\, \nrmK{\beta\cdot\nabla v}{q}{K}
    \\
    &\lsim |K|^{\frac1p - \frac1q}
    \Big( |M_K^{-1}|\,|\beta|\, \nrmK{v}{q}{K}
    + \nrmK{\beta\cdot\nabla v}{q}{K} \Big)\, f_{q,K}(c,s) .
  \end{align*}
  Combining with \eqref{eq:f2-dbound}, again applied with $q$ in
  place of its generic exponent $p$, and noting that the powers of
  $|K|$ cancel,
  \begin{equation*}
    f_{q,K}(c,s)^2
    \lsim \frac{|M_K|}{|\beta|}
    \Big( |M_K^{-1}|\,|\beta|\, \nrmK{v}{q}{K}
    + \nrmK{\beta\cdot\nabla v}{q}{K} \Big)\, f_{q,K}(c,s).
  \end{equation*}
  Canceling one factor of
  $f_{q,K}(c,s)$  and using \eqref{eq:MhK} and
  \eqref{eq:shapereg},
  \begin{equation}
    \label{eq:perp-bound}
    \nrmK{\varphi^\perp}{q}{K}
    \lsim  \nrmK{v}{q}{K}
    + \frac{h_K}{|\beta|}\,
    \nrmK{\beta\cdot\nabla v}{q}{K}
    \lsim \tnrm{v}{q,K} .
  \end{equation}
  Under the additional
  hypothesis~\eqref{eq:annihilated},
   the term $\nrmK{v}{q}{K}$ that entered through
   $\ip{\beta\cdot\nabla w}{v}{K}$
   is absent, so  the
  above simplifies to
  \begin{equation}
    \label{eq:perp-bound-refined}
    \nrmK{\varphi^\perp}{q}{K}
    \lsim a_K \, \nrmK{\beta\cdot\nabla v}{q}{K} .
  \end{equation}

  \emph{Step 2.} Next we bound the component $\mathring\varphi$.  If
  $k = 1$ then $\mathring \varphi = 0$. So suppose $k \ge 2$ and
  $z \not\equiv 0$ in \eqref{eq:Pi-v-split} so $\mathring \varphi = \bint z \not\equiv 0$. The Fortin condition
  \eqref{eq:F2} with test function $z$ gives
  \begin{equation*}
    \ip{z}{\bint z}{K}
    = \ip{z}{v - \varphi^\perp}{K}
    \le \nrmK{z}{p}{K}
    \big( \nrmK{v}{q}{K} + \nrmK{\varphi^\perp}{q}{K} \big) .
  \end{equation*}
  By Lemma~\ref{lem:bubble}, applied with $j=k-2$, $r=q$, and
  $r'=p$,
  $\nrmK{\bint z}{q}{K} \nrmK{z}{p}{K} \le C \ip{z}{\bint z}{K}$, so
  dividing by $\nrmK{z}{p}{K}$ and using \eqref{eq:perp-bound},
  \begin{equation}
    \label{eq:int-bound}
    \nrmK{\mathring\varphi}{q}{K}
    = \nrmK{\bint z}{q}{K}
    \lsim   \nrmK{v}{q}{K} + \nrmK{\varphi^\perp}{q}{K}
    \lsim \tnrm{v}{q,K} .
  \end{equation}
  Under the additional hypothesis~\eqref{eq:annihilated},
  the term $\nrmK{v}{q}{K}$ that arose from
  $\ip{z}{v}{K}$
  vanishes, so $\ip{z}{\bint z}{K} = -\ip{z}{\varphi^\perp}{K}$ and,
  by~\eqref{eq:perp-bound-refined},
  \begin{equation}
    \label{eq:int-bound-refined}
    \nrmK{\mathring\varphi}{q}{K}
    \lsim \nrmK{\varphi^\perp}{q}{K}
    \lsim a_K\, \nrmK{\beta\cdot\nabla v}{q}{K}.
  \end{equation}

  \emph{Step 3.}
  By \eqref{eq:perp-bound} and \eqref{eq:int-bound},
  $\nrmK{\varphi}{q}{K} \le C \tnrm{v}{q,K}$. For the streamline
  derivative, since $\varphi \in P_{k+N}(K)$, the inverse estimate
  \eqref{eq:inverse} together with \eqref{eq:MhK} and
  \eqref{eq:shapereg} gives
  \begin{equation*}
    \frac{h_K}{|\beta|} \nrmK{\beta\cdot\nabla\varphi}{q}{K}
    \le h_K \nrmK{\nabla \varphi}{q}{K}
    \le C\, h_K |M_K^{-1}|\, \nrmK{\varphi}{q}{K}
    \le C \kappa_0 \nrmK{\varphi}{q}{K}
    \le C\, \tnrm{v}{q,K},
  \end{equation*}
  so the local
  bound~\eqref{eq:varPi-local-bound} follows.
  Raising it to the $q$-th power
  and summing over $K \in \oh$ proves the global one.
  Finally, under~\eqref{eq:annihilated}, adding
  \eqref{eq:perp-bound-refined} and \eqref{eq:int-bound-refined}
  gives the first bound of~\eqref{eq:varPi-refined}, and applying to it
  the same argument using the
  inverse inequality~\eqref{eq:inverse}, gives the
  second.
\end{proof}

\subsection{Augmented Fortin operators}
\label{sec:augmented}

The derivative bound implicit in Theorem~\ref{thm:main} carries the
weight $a_K$, i.e., on each element, for any $v \in \W{q}{K}$,
\begin{equation}
  \label{eq:oldbounds}
  \begin{aligned}
    \nrmK{\varPi_K v}{q}{K}
    & \le C \big( \nrmK{v}{q}{K} + a_K \nrmK{\beta\cdot\nabla v}{q}{K}
    \big),
    \\
    \nrmK{\beta\cdot\nabla \varPi_K v}{q}{K}
    & \le C a_K^{-1} \big( \nrmK{v}{q}{K}
    + a_K \nrmK{\beta\cdot\nabla v}{q}{K} \big),
  \end{aligned}
\end{equation}
and the term $a_K^{-1} \nrmK{v}{q}{K}$ reflects the mapping of
constants by  $\varPi_K$
 to nonconstant functions, whose streamline derivatives scale
like $a_K^{-1}$ times their size. Adapting a correction trick from
the proof of~\cite[Lemma~3.2]{GopalQiu14} (see also~\cite[Theorems~5.4 and~5.7]{DPGacta} and \cite{FuhreHeuer24}), we
now improve~\eqref{eq:oldbounds} to~\eqref{eq:kone-bound} below,  at the price of augmenting
the test space by one constant per element.

To show the idea in general, for any linear operator
$R_K : \W{q}{K} \to \W{q}{K}$, define
\begin{equation}
  \label{eq:piplus}
  \varPi^+_K v \,:=\, R_K v + \varPi_K ( v - R_K v ),
  \qquad
  (\varPi^+_h v)|_K := \varPi^+_K (v|_K) .
\end{equation}
Two properties are immediate from the identity
$\varPi^+_K v - v = ( \varPi_K - I )( v - R_K v )$.
First, $\varPi^+_K$ satisfies the same
Fortin conditions~\eqref{eq:fortin-conditions} as $\varPi_K$
and the Fortin property of
Lemma~\ref{lem:fortin-property} holds verbatim for $\varPi^+_h$. Second, $\varPi^+_K$ maps into the
augmented space $V(K) + \operatorname{range}(R_K)$ and, if $R_K$ is
a projection,
\begin{equation}
  \label{eq:piplus-reproduces}
  \varPi^+_K \psi = \psi \qquad \text{ for every } \psi \in \operatorname{range}(R_K).
\end{equation}
Now,  for $k = 1$ the conditions~\eqref{eq:annihilated} on $v$ say exactly that
$v$ has mean zero: the space $P_{k-2}(K)$ is trivial, and by definition~\eqref{eq:jN-pols},
$\beta\cdot\nabla \Pkperp{1}(K) = \beta\cdot\nabla P_1(K)$ consists
of constants. This motivates the choice of $R_K$ in the next result.

\begin{theorem}[Constant-preserving Fortin operator]
  \label{thm:kone}
  Adopt the hypotheses of Theorem~\ref{thm:main}. On each
  element let $R_K v := \bar v$, the mean value of $v$ over $K$,
  and define $\varPi^+_K$ by \eqref{eq:piplus}. Then the following
  statements hold.
  \begin{enumerate}[label=(\alph*)]
  \item For every $k \ge 1$, the operator $\varPi^+_K$ maps
    $\W{q}{K}$ into
    \begin{equation}
      \label{eq:Vplus}
      V^+(K) := V(K) + P_0(K),
    \end{equation}
    satisfies the Fortin conditions \eqref{eq:fortin-conditions},
    and reproduces constants, $\varPi^+_K c = c$ for all
    $c \in P_0(K)$. The sum in \eqref{eq:Vplus} is direct whenever
    $N \ge 2$.
  \item If $k = 1$ (in which case Assumption~\ref{asm:theta} is
    \emph{not} needed), then there is a
    $C = C(N, q, \kappa_0)$
     such that for all
     $K \in \oh$ and $v \in \W{q}{K}$,
     \begin{equation}
       \label{eq:kone-bound}
       \begin{aligned}
         \nrmK{\varPi^+_K v}{q}{K}
         & \le C \big( \nrmK{v}{q}{K}
           + a_K \nrmK{\beta\cdot\nabla v}{q}{K} \big),
         \\
         \nrmK{\beta\cdot\nabla \varPi^+_K v}{q}{K}
         & \le C \nrmK{\beta\cdot\nabla v}{q}{K}.
       \end{aligned}
     \end{equation}
     Consequently, for every choice of
     per-element weights $\sigma_K \ge a_K$,
     \begin{equation}
       \label{eq:kone-sigma}
       \sigma_K^{-1} \nrmK{\varPi^+_K v}{q}{K}
       + \nrmK{\beta\cdot\nabla \varPi^+_K v}{q}{K}
       \,\le\, C \big( \sigma_K^{-1} \nrmK{v}{q}{K}
       + \nrmK{\beta\cdot\nabla v}{q}{K} \big).
     \end{equation}
  \item For every $k\ge1$, there is a
    $C=C(N,k,q,\theta_0,\kappa_0)$, independent of $\theta_0$ when
    $k=1$, such that
    \begin{equation}
      \label{eq:piplus-bound}
      \tnrm{\varPi_K^+v}{q,K}
      \le C\tnrm{v}{q,K},
      \qquad v\in\W{q}{K}.
    \end{equation}
    Consequently,
    $\tnrm{\varPi_h^+v}{q,\oh}\le
    C\tnrm{v}{q,\oh}$ for all $v\in\W{q}{\oh}$.
  \end{enumerate}
\end{theorem}

\begin{proof}
  \emph{Proof of (a).}
  The first statement of (a) is contained in the two properties noted
  after~\eqref{eq:piplus}, $R_K$ being a projection onto $P_0(K)$.
  For  the directness of the sum~\eqref{eq:Vplus}, suppose $\bdr\, u + \bint\, z$ is a nonzero
  constant $c$, with $u \in \Pkperp{k}(K)$ and
  $z \in P_{k-2}(K)$. Restricting to $\partial K$, where $\bint$
  vanishes, gives $\bdr\, u = c$ there. Choose a non-characteristic
  facet $F_i$, which exists by
  Lemma~\ref{lem:two-noncharacteristic}. On $F_i$ we have
  $\bdr|_{F_i} = s_i b_i$, and $b_i$ vanishes on the relative boundary
  $\partial F_i$, which is nonempty because $N \ge 2$. Hence
  $c = 0$, a contradiction.

  \emph{Proof of (b).}
  For \eqref{eq:kone-bound}, set $g := v - \bar v$. Then
  $\ip{\beta\cdot\nabla w}{g}{K} = 0$ for all
  $w \in \Pkperp{1}(K) = P_1(K)$. The
  $P_{-1}(K)$-condition in \eqref{eq:annihilated} is vacuous. So the
  refined bounds \eqref{eq:varPi-refined} of Theorem~\ref{thm:main}
  apply to $g$ and give, using
  $\beta\cdot\nabla g = \beta\cdot\nabla v$,
  \begin{equation*}
    \nrmK{\varPi_K( v - \bar v)}{q}{K}
    \le C a_K \nrmK{\beta\cdot\nabla v}{q}{K},
    \qquad
    \nrmK{\beta\cdot\nabla \varPi_K(v - \bar v)}{q}{K}
    \le C \nrmK{\beta\cdot\nabla v}{q}{K} .
  \end{equation*}
  Since $\nrmK{\bar v}{q}{K} \le \nrmK{v}{q}{K}$ by Jensen's
  inequality and $\beta\cdot\nabla \bar v = 0$, the bounds
  \eqref{eq:kone-bound} follow from~\eqref{eq:piplus}. Multiplying the first bound in
  \eqref{eq:kone-bound} by $\sigma_K^{-1}$ and using
  $a_K/\sigma_K \le 1$ gives \eqref{eq:kone-sigma}.

  \emph{Proof of (c).}
  From \eqref{eq:piplus} and Theorem~\ref{thm:main},
  \begin{equation*}
    \tnrm{\varPi_K^+v}{q,K}
    \le \nrmK{\bar v}{q}{K}
      +C\tnrm{v-\bar v}{q,K}
    \le C\tnrm{v}{q,K},
  \end{equation*}
  where the last inequality follows from Jensen's inequality and
  $\beta\cdot\nabla\bar v=0$. Raising this estimate to the $q$-th
  power and summing over the mesh proves the global bound.
\end{proof}

\begin{remark}[Benign one-dimensional exclusion]
  \label{rem:N1}
  The restriction to $N \ge 2$ in the directness assertion of
  Theorem~\ref{thm:kone}(a) is not an artifact: for $N = 1$ and
  $k = 2$, with $K = (0,1)$ and $\beta > 0$, observe that
  $\bdr = 2x-1$ and
  $
    \bdr \cdot (2x-1) + 4\, \bint
    = (2x-1)^2 + 4x(1-x) = 1 ,
  $
  where $2x-1 \in \Pkperp{2}(K)$.
  So $V(K)$ already contains
  the constants and $V^+(K) = V(K)$.
\end{remark}

One further augmentation will be helpful.  Define
\begin{align}
  \label{eq:Vplusplus}
  V^{++}(K)
  & :=V^+(K)+\operatorname{span}\{\bint\}
    \subset P_{k+N}(K),
  &&
     V_h^{++}:=\bigtimes_{K\in\oh}V^{++}(K),
  \\
  \label{eq:piplusplus}
  \varPi_K^{++}v
  & :=\varPi_K^+v+
    \frac{\ip{v-\varPi_K^+v}{1}{K}}{\ip{\bint}{1}{K}}\,\bint,
  && (\varPi_h^{++} v)|_K = \varPi_K^{++} (v|_K).
\end{align}

\begin{theorem}[Fortin operator that preserves both constants and means]
  \label{thm:kone-plusplus}
  Adopt the hypotheses and notation of Theorem~\ref{thm:kone}.
  Then the following statements hold.
  \begin{enumerate}[label=(\alph*)]
  \item For every $k\ge1$, the operator $\varPi_K^{++}$ maps
    $\W{q}{K}$ into $V^{++}(K)$, satisfies the Fortin conditions
    \eqref{eq:fortin-conditions}, reproduces constants, and preserves mean: for $c\in P_0(K)$,
    \begin{equation}
      \label{eq:F3}
      \varPi_K^{++}c=c,
      \qquad
      \ip{c}{\varPi_K^{++}v-v}{K}=0.
    \end{equation}
    The global Fortin property of Lemma~\ref{lem:fortin-property}
    holds with $\varPi_h^{++}$. Moreover, if $k\ge2$, then
    \begin{equation*}
      \varPi_K^{++}=\varPi_K^+,
      \qquad V^{++}(K)=V^+(K),
    \end{equation*}
    while if $k=1$ and $N\ge2$, then
    $
      V^{++}(K)
      =V(K)\oplus P_0(K)\oplus\operatorname{span}\{\bint\}.
    $
  \item If $k=1$, the estimates \eqref{eq:kone-bound} and
    \eqref{eq:kone-sigma} remain valid with $\varPi_K^{++}$ in
    place of $\varPi_K^+$.
  \item For $k\ge1$, there is a
    $C=C(N,k,q,\theta_0,\kappa_0)$, independent of $\theta_0$ when
    $k=1$, such that
    \begin{equation}
      \label{eq:piplusplus-bound}
      \tnrm{\varPi_K^{++}v}{q,K}
      \le C\tnrm{v}{q,K},
      \qquad v\in\W{q}{K}.
    \end{equation}
    Consequently,
    $\tnrm{\varPi_h^{++}v}{q,\oh}\le
    C\tnrm{v}{q,\oh}$ for all $v\in\W{q}{\oh}$.
  \end{enumerate}
\end{theorem}

\begin{proof}
  \emph{Proof of (a).}
  Since $\bint$ vanishes on $\partial K$,
  $
    \ap{D_Kw,\bint}=0
    $
    for all $w$ in $\Pk(K)$.
  Thus the bubble correction in \eqref{eq:piplusplus} does not
  change \eqref{eq:F1}. If $k=1$, condition \eqref{eq:F2} is
  vacuous. If $k\ge2$, then $1\in P_{k-2}(K)$, and \eqref{eq:F2}
  for $\varPi_K^+$ shows that the correction coefficient vanishes.
  Hence $\varPi_K^{++}=\varPi_K^+$, and also
  $V^{++}(K)=V^+(K)$ because
  $\bint\in\bint P_{k-2}(K)\subset V(K)$.
  The definition \eqref{eq:piplusplus} gives the moment condition in
  \eqref{eq:F3}. If $v=c\in P_0(K)$, then $\varPi_K^+c=c$ by
  Theorem~\ref{thm:kone}(a), so the correction vanishes and
  $\varPi_K^{++}c=c$.
  The element-wise Fortin conditions~\eqref{eq:fortin-conditions} obviously
  imply the global Fortin
  property.
  It remains to verify the direct sum asserted when $k=1$ and
  $N\ge2$. Suppose
  $
    \bdr u+c+\bint d=0,
    $
  for some $ u\in P_1(K)$ and  $c,d\in\R.$
  Restricting first to a non-characteristic facet and then to its
  relative boundary gives $c=0$. It follows that $\bdr u$ vanishes
  on $\partial K$, so the boundary argument in
  Lemma~\ref{lem:fortin-def} gives $u=0$. Finally, $\bint d=0$ gives
  $d=0$, thus proving that the sum is direct.

  \emph{Proof of (b).}
  Let $k=1$, set $g:=v-\bar v$, and denote the bubble correction by
  $t_K:=\varPi_K^{++}v-\varPi_K^+v$. Since
  $\ip{g}{1}{K}=0$, \eqref{eq:piplus} gives
  $
    t_K=-\ip{\varPi_Kg}{1}{K} \,\bint    /{\ip{\bint}{1}{K}}.
  $
  Lemma~\ref{lem:bubble}, applied with $j=0$, $r=q$, $r'=p$, and
  $z=1$, gives
  \[
    \nrmK{\bint}{q}{K}\nrmK{1}{p}{K}
    \lsim \ip{\bint}{1}{K}.
  \]
  Therefore, H\"older's inequality and
  \eqref{eq:varPi-refined} for the mean-zero function $g$ yield
  \begin{equation*}
    \nrmK{t_K}{q}{K}
    \lsim \nrmK{\varPi_Kg}{q}{K}
    \lsim a_K\nrmK{\beta\cdot\nabla v}{q}{K}.
  \end{equation*}
  Since $t_K$ is a polynomial of fixed degree, the inverse estimate
  \eqref{eq:inverse} and shape regularity give
  \begin{equation*}
    \nrmK{\beta\cdot\nabla t_K}{q}{K}
    \lsim a_K^{-1}\nrmK{t_K}{q}{K}
    \lsim \nrmK{\beta\cdot\nabla v}{q}{K}.
  \end{equation*}
  Adding these two estimates to \eqref{eq:kone-bound} proves that
  bound for $\varPi_K^{++}$. If $\sigma_K\ge a_K$, the same
  estimates, with $a_K/\sigma_K\le1$, also give
  \eqref{eq:kone-sigma}.

  \emph{Proof of (c).}
  For $k=1$, the assertion follows from part~(b). For $k\ge2$,
  part~(a) gives $\varPi_K^{++}=\varPi_K^+$, so both the local and
  global assertions follow from Theorem~\ref{thm:kone}(c).
\end{proof}

\section{Quasioptimality of the practical DPG method}
\label{sec:quasiopt}

Section~\ref{sec:weakform} introduced the practical method
\eqref{eq:method}, computed with any finite dimensional test space
$Y_h$ satisfying~\eqref{eq:Yh-condition}. In this
section we place it in a functional-analytic setting adapted to the
continuous problem~\eqref{eq:advection}, use the Fortin operator to
prove that it is quasioptimal, and derive convergence rates in~$h$.
Throughout this section we fix a Lebesgue exponent
$  1 < p < \infty
$
for the \emph{trial} side, so that its conjugate $q$ of
\eqref{eq:conjugate}  is the exponent
of the \emph{test} side.

\subsection{The undiscretized problem}

Since $\dive \beta = 0$, no coercivity is available from a reaction
term, so the solvability of the non-coercive transport
problem~\eqref{eq:advection}, as well as the trace theory needed to
give its boundary condition meaning, is  secured by conditions
on the flow. We assume the standard ones (see, e.g.,
\cite{BroerDahmeSteve18,DemkoGopal10,MugaTylerZee19}). Here
$
  \Gamma_{\mathrm{out}} := \{x \in \partial\Om : \beta\cdot n > 0\}
$
denotes the outflow boundary, complementing the inflow boundary
$\Gamma_{\mathrm{in}}$ of \eqref{eq:advection}.

\begin{assumption}[Flow conditions on $\Om$] \label{asm:fill}
  \textup{(i)~Separated inflow and outflow:} The  sets
    $\overline{\Gamma}_{\mathrm{in}}$ and
    $\overline{\Gamma}_{\mathrm{out}}$ are disjoint.
    \textup{(ii)~$\Om$-filling flow:} There exist
    $z_+, z_- \in L_\infty(\Om)$ with
    $\beta\cdot\nabla z_\pm \in L_\infty(\Om)$, and a constant
    $c_0 > 0$, such that
    \begin{equation}
      \label{eq:filling}
      - \beta \cdot \nabla z_\pm = c_0 \ \text{ in } \Om,
      \qquad
      z_+ = 0 \ \text{ on } \Gamma_{\mathrm{out}},
      \qquad
      z_- = 0 \ \text{ on } \Gamma_{\mathrm{in}} .
    \end{equation}
\end{assumption}

Neither part is new. Part (i) is the well-separation condition
of~\cite[eq.~(9)]{Canti17} (see also
\cite[Assumption~2.2]{MugaTylerZee19}). Few results are known without it (an exception being
\cite[Lemma~4.1]{GopalMonkSepul15}).
Part~(ii) is the
$\Om$-filling condition of
\cite[Assumption~2.8]{MugaTylerZee19}. It states informally
that the flow of $\beta$ has no closed streamlines and that every
streamline joins $\Gamma_{\mathrm{in}}$ to $\Gamma_{\mathrm{out}}$
in finite time (indeed $z_+/c_0$ is the time needed to reach
$\Gamma_{\mathrm{out}}$) and it holds, e.g., whenever the method of
characteristics applies
(see~\cite[Remark~2.9]{MugaTylerZee19} and \cite{DahmeHuangSchwa11a}).
The constant $c_0$
in~\eqref{eq:filling} is immaterial: rescaling $z_\pm$ changes it at
will. Another condition in the literature, perhaps
easier to verify, is the potential condition
of~\cite[Section~2.4]{Canti17}. In our case the reaction
coefficient $\mu$, the divergence $\dive\beta$, and hence the
Friedrichs coefficient
$\sigma_{\beta,\mu;p} = \mu - \tfrac1p\dive\beta$ all vanish, and
that condition reduces to the existence of a
Lipschitz $\psi = -\zeta$ with $\beta\cdot\nabla\psi \ge 2\mu_0 > 0$,
from which \eqref{eq:filling} follows by taking $z_+$ to be the
transit time to $\Gamma_{\mathrm{out}}$.
Note that Assumption~\ref{asm:fill}
rules out $\beta$ with closed streamlines in $\Om$ that never meet
$\Gamma_{\mathrm{in}}$, i.e., periodic orbits of $x'(t) = \beta(x(t))$ with period $T>0$: along it, $d z_+(x(t)) / dt = \beta \cdot \nabla z_+ = - c_0$, so $z_+$ strictly decreases at a nonzero constant rate, yet $z_+(x(T)) = z_+(x(0))$ by periodicity, which is impossible.
Note that no quantitative stability constant
of the continuous problem will enter Theorem~\ref{thm:quasiopt} or
Corollary~\ref{cor:rates} below; the constant $C_{\mathrm{PF}}(q,\beta)$ is
used only in Corollary~\ref{cor:L2}, where we convert energy control
into mesh-independent $L_p(\Om)$ control.

We need the
unbroken graph spaces on~$\Om$: for $1 < r < \infty$,
\begin{equation*}
  \W{r}{\Om} = \{ v \in L_r(\Om) :
  \beta\cdot\nabla v \in L_r(\Om) \}.
\end{equation*}
Under Assumption~\ref{asm:fill}(i), functions in
$\W{r}{\Om}$ have traces on
$\partial\Om$ in the $|\beta\cdot n|$-weighted $L_r$ sense,
smooth functions are dense in $\W{r}{\Om}$, and the
integration by parts formula extends to
\begin{equation}
  \label{eq:int-by-parts-graphspace}
  \int_\Om \beta \cdot \nabla v \; w + \int_\Om v \beta \cdot \nabla w = \int_{\partial \Om} \beta \cdot n \, v \, w,
  \qquad v \in \W{p}{\Om}, \; w \in \W{q}{\Om}
\end{equation}
(see e.g.,
\cite[Lemmas~2.1--2.2]{Canti17},
\cite[Remark~2.5, Lemma~2.6]{MugaTylerZee19}
and~\cite{BroerDahmeSteve18,DemkoGopal10,Jense04}). Let
\begin{equation*}
  W_{r,0}(\partial_\beta, \Om) := \{ v \in \W{r}{\Om} :
  v|_{\Gamma_{\mathrm{in}}} = 0 \}
\end{equation*}
denote the subspace incorporating the inflow boundary condition of
\eqref{eq:advection} in that trace sense.

\begin{lemma}[Wellposedness and curved Poincar\'e--Friedrichs
  inequality]
  \label{lem:transport}
  Suppose  $1 < r < \infty$, Assumptions~\ref{asm:beta}
  and~\ref{asm:fill} hold, and
  $C_{\mathrm{PF}}(r,\beta) := r \, \max( \|z_+\|_{L_\infty(\Om)},
  \|z_-\|_{L_\infty(\Om)} ) / c_0$. Then
  \begin{equation}
    \label{eq:curvedPF}
    \| v \|_{L_r(\Om)} \le C_{\mathrm{PF}}(r,\beta)
    \| \beta \cdot \nabla v \|_{L_r(\Om)}
    \qquad \text{for all } v \in \W{r}{\Om} \text{ vanishing on }
    \Gamma_{\mathrm{in}} \text{ or } \Gamma_{\mathrm{out}} ,
  \end{equation}
  and for every $f \in L_r(\Om)$ there is a unique
  $u \in W_{r,0}(\partial_\beta,\Om)$ with
  $\beta \cdot \nabla u = f$.
\end{lemma}

\begin{proof}
  The inequality \eqref{eq:curvedPF} follows from
  \cite[Lemma~2.10]{MugaTylerZee19} and  unique solvability from
  \cite[Theorem~A and Remark~3.1]{MugaTylerZee19}.
\end{proof}
We first complete the prescription of the method \eqref{eq:method}
by specifying the norm whose dual appears there. As in \eqref{eq:trfunctional}, let
$Y = \W{q}{\oh}$. Observe that  the
spaces of \eqref{eq:Yh-condition} form the chain
$
  V_h \,\subseteq\, Y_h \,\subset\, Y .
$
Equip $Y$ with the
$\beta$-weighted test norm obtained by rescaling the norm
\eqref{eq:norm} element by element. Recall that
$a_K=h_K/|\beta_K|$. Then
\begin{equation}
  \label{eq:Ynorm}
  \| v \|_{Y(K)}^q
  := a_K^{-q} \, \nrmK{v}{q}{K}^q
  + \nrmK{\beta\cdot\nabla v}{q}{K}^q,
  \qquad
  \| v \|_Y^q := \sum_{K \in \oh} \| v\|_{Y(K)}^q .
\end{equation}
Linear
functionals $g$ on $Y_h$ are normed by
\begin{equation}
  \label{eq:dualnorm}
  \| g \|_{Y_h'} := \sup_{0\neq v \in Y_h}
  \frac{ g(v)}{\|v\|_Y},
\end{equation}
which is the norm used in the specification of the method~\eqref{eq:method}.

Next we set the undiscretized trial space, realizing
interface variables as functionals on $Y$, in the spirit of
\cite{CarstDemkoGopal16}. Formula~\eqref{eq:trfunctional} defines
$\Dh : \W{p}{\Om} \to Y'$ by the same element-wise sum (where
the right hand side is well defined even when  $z$ is
taken to be in $\W{p}{\Om}$ by virtue of~\eqref{eq:DK}).
Set
\begin{equation}
  \label{eq:Xdef}
  \hat X := \Dh\big( \W{p}{\Om} \big),
  \qquad
  \hat{X}_0 :=
  \Dh\big( W_{p,0}(\partial_\beta,\Om) \big),
  \qquad
  X := L_p(\Om) \times \hat{X}_0.
\end{equation}
The definitions \eqref{eq:trial} now give the  inclusions
$\Xhk\subset\hat X$, $\Xhok\subset\hat{X}_0$, and hence
$X_h\subset X$.
The DPG bilinear form~\eqref{eq:bform} extends to $X \times Y$ as
\begin{equation}
  \label{eq:extended-b}
  \bform\big( (w, \mu),\, v \big)
  = \ap{ \mu, v} - \sum_{K\in \oh}(w, \,\beta\cdot\nabla v)_{K},
  \qquad
  \text{ for any }\cl w = (w, \mu) \in X, \; v \in Y.
\end{equation}

\begin{lemma}[Injectivity]
  \label{lem:injectivity}
  Suppose Assumptions~\ref{asm:beta} and \ref{asm:fill} hold. If
  $\cl w = (w, \mu) \in X$ satisfies $\bform(\cl w, v) = 0$
  for all $v \in Y$, then $\cl w = 0$.
\end{lemma}
\begin{proof}
  We use the space of (smooth compactly supported) Schwartz test
  functions on $K$ and $\Om$, denoted by $C^\infty_c(K)$ and
  $C^\infty_c(\Om)$, respectively.
  By \eqref{eq:Xdef}, $\mu = \Dh z$ for some
  $z \in W_{p,0}(\partial_\beta,\Om)$.
  Fix $K \in \oh$ and let $v \in C^\infty_c(K)$, extended by
  zero to an element of $Y$. Then, since
  $\ap{D_K z, v} = 0$, we have   $0 = \bform(\cl w, v) =
  -\ip{w}{\beta\cdot\nabla v}{K}$ for all $v \in C_c^\infty(K)$.
  Since $\beta_K$ is constant, this implies that the distributional
  derivative $\beta\cdot\nabla w$ vanishes on $K$. In particular
  $w|_K \in \W{p}{K}$ with
  \begin{equation}
    \label{eq:betanablawK}
    \beta\cdot\nabla w|_K = 0, \quad \text{ for every }K \in \oh.
  \end{equation}
  Using this and~\eqref{eq:DK},
  $
    -\ip{w}{\beta\cdot\nabla v}{K}
    = \ip{\beta \cdot \nabla w}{v}{K} - \ap{D_K w, v}
    = - \ap{D_K w, v},
  $ so
  \begin{equation}
    \label{eq:inj-step2}
    0 = \bform(\cl w, v)
    = \ap{D_h ( z - w ), v},
    \qquad \text{ for all } v \in Y.
  \end{equation}

  Let $\zeta := z - w \in \W{p}{\oh}$.  The distributional derivative
  $\beta\cdot\nabla \zeta$ coincides with the element-wise one: indeed,
  for any $v \in C_c^\infty(\Om)$, we have
  $(\beta\cdot\nabla \zeta)(v) = - (\zeta, \beta\cdot \nabla v)_\Om =
  (\beta \cdot \nabla \zeta, v)_\Om$ using~\eqref{eq:DK}
  and~\eqref{eq:inj-step2}. By~\eqref{eq:int-by-parts-graphspace}
  and~\eqref{eq:inj-step2}, now taken for $v \in \W{q}{\Om} \subset Y$,
  $
    0 = \ip{\beta\cdot\nabla \zeta}{v}{\Om}
    + \ip{\zeta}{\beta\cdot\nabla v}{\Om}
    = \int_{\partial\Om} (\beta \cdot n)\, \zeta\, v,
  $
  so the trace of $\zeta$ must vanish on $\Gamma_{\mathrm{in}}$.

  Consequently $w = z - \zeta \in W_{p,0}(\partial_\beta,\Om)$ (since the  trace of both $z$ and $\zeta$ on
  $\Gamma_{\mathrm{in}}$ vanishes)
  and $\beta\cdot\nabla w = 0$. The uniqueness part of
  Lemma~\ref{lem:transport} now implies $w = 0$. Returning to
  \eqref{eq:inj-step2} with $w = 0$, we find
  $\ap{\mu, v} = \ap{D_h z, v} = 0$ for all
  $v \in Y$, i.e., $\mu = 0$.
\end{proof}

\subsection{Stability of the practical DPG method}

Since $X_h=P_{k-1}(\oh)\times\Xhok\subset
P_{k-1}(\oh)\times\Xhk$, Lemma~\ref{lem:fortin-property} gives
\begin{equation}
  \label{eq:fortin-Xh}
  \bform(\cl w_h, v - \varPi_h v) = 0,
  \qquad \cl w_h \in X_h, \quad v \in Y,
\end{equation}
the key property that will get us stability.
Define the ``energy norm'' and a discrete counterpart, both for any
$\cl w \in X$, by
\begin{equation}
  \label{eq:energy}
  \| \cl w \|_{\Enrm} :=
  \sup_{0 \neq v \in Y} \frac{\bform(\cl w, v)}{\|v\|_Y},
  \qquad
  \| \cl w \|_{\Enrm,h} :=
  \sup_{0 \neq v_h \in Y_h} \frac{\bform(\cl w, v_h)}{\|v_h\|_Y}.
\end{equation}
Since
$Y_h \subset Y$,
\begin{equation}
  \label{eq:Eh-le-E}
  \|\cl w\|_{\Enrm,h} \le \|\cl w\|_{\Enrm}.
\end{equation}
By Lemma~\ref{lem:injectivity}, $\| \cdot \|_{\Enrm}$ is a norm
on $X$. That $\|\cdot\|_{\Enrm,h}$ is a norm on
$X_h$ follows by the Fortin property:
if $\cl x_h \in X_h$ satisfies
$\| \cl x_h \|_{\Enrm,h} = 0$, then $0 = \bform(\cl x_h, \varPi_h v) = \bform(\cl x_h, v)$ for all $v \in Y$
by~\eqref{eq:fortin-Xh}, so $\cl x_h = 0$ by
Lemma~\ref{lem:injectivity}.
We begin with a quick observation that the Fortin operator is bounded also in the rescaled test norm~\eqref{eq:Ynorm}.


\begin{corollary}
  \label{cor:fortin-Y}
  In the setting of Theorem~\ref{thm:main},
  there is a $C_\varPi = C_\varPi(N, k, q, \theta_0, \kappa_0)$,
  independent of $\theta_0$ when $k = 1$, such that
  $\| \varPi_h v \|_Y \le C_\varPi \| v \|_Y$ for all $v \in Y$.
  The same conclusion holds for the mean-preserving operator:
  there is a constant $C_{\varPi^{++}}$ with the same dependencies
  such that
  \begin{equation}
    \label{eq:piplusplus-Y-bound}
    \|\varPi_h^{++}v\|_Y
    \le C_{\varPi^{++}}\|v\|_Y,
    \qquad v\in Y.
  \end{equation}
\end{corollary}
\begin{proof}
  Multiplying the local bound
  $\tnrm{\varPi_K v}{q,K} \le C \tnrm{v}{q,K}$ of
  Theorem~\ref{thm:main} by $a_K^{-1}$ gives
  \begin{equation*}
    a_K^{-1}\nrmK{\varPi_K v}{q}{K}
    + \nrmK{\beta \cdot \nabla \varPi_K v}{q}{K}
    \,\le\, C \Big( a_K^{-1} \nrmK{v}{q}{K}
    + \nrmK{\beta\cdot\nabla v}{q}{K} \Big) ,
  \end{equation*}
  so $\|\varPi_K v\|_{Y(K)} \le 2^{1/p}\, C\, \|v\|_{Y(K)}$ by the
  equivalence $(a^q + b^q)^{1/q} \le a + b \le 2^{1-1/q}(a^q +
  b^q)^{1/q}$ of the $\ell_1$ and $\ell_q$ norms on $\R^2$. Raise to
  the $q$-th power, sum over $K$, and set
  $C_\varPi := 2^{1/p}\, C$.
  Applying the same argument to \eqref{eq:piplusplus-bound} proves
  \eqref{eq:piplusplus-Y-bound}.
\end{proof}

Let $u \in W_{p,0}(\partial_\beta,\Om)$ solve \eqref{eq:advection}
(Lemma~\ref{lem:transport}) for the data $f$ of $\ell$
in~\eqref{eq:method}. The exact pair
$
  \cl u := (u, \Dh u) \in X
$
is consistent: by~\eqref{eq:DK},
\begin{equation}
  \label{eq:consistency}
  \bform( \cl u, v )
  = \sum_{K \in \oh} \big( -\ip{u}{\beta\cdot\nabla v}{K}
  + \ip{\beta\cdot\nabla u}{v}{K} + \ip{u}{\beta\cdot\nabla v}{K}
  \big)
  = \ip{f}{v}{\Om} = \ell(v)
\end{equation}
for all $v \in Y$. By \eqref{eq:consistency},
$\| \ell - \bform(\cl w_h, \cdot) \|_{Y_h'} = \| \cl u -
\cl w_h \|_{\Enrm,h}$, so \eqref{eq:method} is equivalent to
\begin{equation}
  \label{eq:method-Eh}
  \| \cl u - \cl u_h \|_{\Enrm,h}
  = \min_{\cl w_h \in X_h}
  \| \cl u - \cl w_h \|_{\Enrm,h}.
\end{equation}
The existence of this minimizer and its quasioptimality in $\| \cdot \|_{\Enrm}$ is proved next.

\begin{theorem}[Quasioptimality]
  \label{thm:quasiopt}
  Let $1 < p < \infty$. Suppose Assumptions~\ref{asm:beta} and
  \ref{asm:fill} hold,   $Y_h$ satisfy~\eqref{eq:Yh-condition},
  the shape regularity bound
  \eqref{eq:shapereg} holds, and that either
  Assumption~\ref{asm:theta} holds or $k = 1$. Let
  $f \in L_p(\Om)$
  and let $C_\varPi$ be the constant of
  Corollary~\ref{cor:fortin-Y}. Then the discrete solution
  $\cl u_h$ of \eqref{eq:method} exists, is unique, and
  satisfies
  \begin{equation}
    \label{eq:quasiopt-E}
    \| \cl u - \cl u_h \|_{\Enrm}
    \,\le\, ( 1 + 2 C_\varPi )
    \inf_{\cl w_h \in X_h}
    \| \cl u - \cl w_h \|_{\Enrm} .
  \end{equation}
\end{theorem}

\begin{proof}
  The map
  $\cl w_h \mapsto \| \ell - \bform(\cl w_h, \cdot)
  \|_{Y_h'} = \| \cl u - \cl w_h\|_{\Enrm, h}$ is
  convex, continuous and coercive on $X_h$,
  so minimizers of~\eqref{eq:method} exist.
  If $\cl u_1, \cl u_2$ are two
  minimizers, their midpoint is also a minimizer, so the triangle
  inequality holds with equality for the two residuals
  $\|\ell - \bform(\cl u_i, \cdot)\|_{Y_h'}$.
  Strict convexity of $\|\cdot \|_{Y_h'}$ (since $\| \cdot \|_Y$
  restricted to $Y_h$ is smooth, being isometric to a subspace of an
  $L_q$ space with $1 < q < \infty$)  then implies  that the two
  residuals must coincide as functionals on $Y_h$, i.e.\
  $\| \cl u_1 - \cl u_2 \|_{\Enrm,h} = 0$, so
  $\cl u_1 = \cl u_2$.
  Thus the discrete solution $\cl u_h$ of \eqref{eq:method} exists and is unique.
  
  To prove quasioptimality, let
  $\delta := \cl u - \cl u_h$ and fix a
  $\cl w_h \in X_h$. For any $v \in Y$, we have  $\varPi_h v \in V_h \subseteq Y_h$ by~\eqref{eq:Yh-condition}.
  Insert
  $\pm \cl w_h$ and $\pm \varPi_h v$ and use
  \eqref{eq:fortin-Xh} on $\cl w_h - \cl u_h \in X_h$ to get
  \begin{align*}
    \bform( \delta, v )
    &= \bform( \cl u - \cl w_h,\, v - \varPi_h v )
    + \bform( \cl w_h - \cl u_h,\, v - \varPi_h v )
    + \bform( \delta,\, \varPi_h v )
    \\
    &= \bform( \cl u - \cl w_h,\, v - \varPi_h v )
    + \bform( \delta,\, \varPi_h v ) .
  \end{align*}
  By the definitions \eqref{eq:energy} and
  Corollary~\ref{cor:fortin-Y},
  \begin{align*}
    \bform( \cl u - \cl w_h, v - \varPi_h v )
    &\le \| \cl u - \cl w_h \|_{\Enrm}\,
    \| v - \varPi_h v \|_Y
    \le ( 1 + C_\varPi )\,
    \| \cl u - \cl w_h \|_{\Enrm}\, \| v \|_Y ,
    \\
    \bform( \delta, \varPi_h v )
    &\le \| \delta \|_{\Enrm,h}\, \| \varPi_h v \|_Y
    \le C_\varPi\, \| \delta \|_{\Enrm,h}\, \| v \|_Y .
  \end{align*}
  Dividing by $\|v\|_Y$ and taking the supremum over $v \in Y$,
  \begin{equation*}
    \| \delta \|_{\Enrm}
    \le ( 1 + C_\varPi ) \| \cl u - \cl w_h \|_{\Enrm}
    + C_\varPi \| \delta \|_{\Enrm,h} .
  \end{equation*}
  Finally, by \eqref{eq:method-Eh} and~\eqref{eq:Eh-le-E}, the last term above  is bounded by
  $
    \| \delta \|_{\Enrm,h}
    \le \| \cl u - \cl w_h \|_{\Enrm,h}
    \le \| \cl u - \cl w_h \|_{\Enrm} ,
  $
  so $\|\delta\|_{\Enrm} \le (1 + 2C_\varPi) \| \cl u - \cl w_h
  \|_{\Enrm}$, thus proving~\eqref{eq:quasiopt-E}.
\end{proof}

\begin{remark}
  The argument of Theorem~\ref{thm:quasiopt} is general and can be
  applied beyond the advection problem. Another general approach
  is the theory of~\cite{MugaZee20}, which we did not apply directly for brevity because it requires us to first prove that the  infinite-dimensional broken Banach
  formulation is  wellposed. If we had applied~\cite[Theorem~4.14]{MugaZee20}, then
 we would have obtained the  quasioptimality constant
  $1 + C_\varPi ( 1 + C_{\mathrm{AO}}(Y) )$, where
  $C_{\mathrm{AO}}(Y) \in [0,1]$ is their asymmetric-orthogonality
  constant of $Y$, instead of
  $1 + 2C_\varPi$.
\end{remark}

\begin{remark}
  In the Hilbert space case  $p = q = 2$, by the standard
  route of~\cite{GopalQiu14}, we
  obtain~\eqref{eq:quasiopt-E} with  $C_\varPi$ instead of
  $1+ 2 C_\varPi$.
\end{remark}

\subsection{Convergence rates}
\label{ssec:rates}

To obtain convergence rates from  Theorem~\ref{thm:quasiopt},
we measure the
best approximation error in a convenient norm on $X$.
Viewing  $\hat X$ as a subspace of $Y'$, define
\begin{equation}
  \label{eq:Xnorm}
  \| (w, \mu) \|_X^p := \nrmK{w}{p}{\Om}^p + \| \mu \|_{Y'}^p.
\end{equation}
\begin{lemma}
  \label{lem:trace-bound}
  Suppose Assumption~\ref{asm:beta} holds. For all
  $z \in \W{p}{\Om}$,
  $(w,\mu) \in X, \ v \in Y$,
  \begin{align}
    \label{eq:Dhz-trace-bound}
    \| \Dh z \|_{Y'}^p
    & \le\,
     \sum_{K\in\oh} \Big(
    \nrmK{z}{p}{K}^p
    + a_K^p \nrmK{\beta\cdot\nabla z}{p}{K}^p \Big) ,
    \\
    \label{eq:continuity}
    \big| \bform\big( (w,\mu), v \big) \big|
    & \le\, 2^{1/q}\, \| (w,\mu) \|_X\, \| v \|_Y.
\end{align}
\end{lemma}
\begin{proof}
  By \eqref{eq:trfunctional}, \eqref{eq:DK}, and the
  H\"older inequality,
  \begin{align*}
    \ap{\Dh z, v}
    &= \sum_{K} \Big[
    \ip{\beta\cdot\nabla z}{v}{K} + \ip{z}{\beta\cdot\nabla v}{K}
    \Big]
    \\
    &\le \sum_K \Big[
    \Big( a_K \nrmK{\beta\cdot\nabla z}{p}{K}
    \Big) \Big( a_K^{-1}\nrmK{v}{q}{K} \Big)
    + \nrmK{z}{p}{K}\, \nrmK{\beta\cdot\nabla v}{q}{K}
    \Big]
    \\
    &\le \bigg( \sum_{K} \Big[
    \nrmK{z}{p}{K}^p
    + a_K^p \nrmK{\beta\cdot\nabla z}{p}{K}^p \Big]
    \bigg)^{1/p} \| v \|_Y .
  \end{align*}
  Continuity of $\bform$ is now immediate:
  bounding the second (volume) term of~\eqref{eq:extended-b}
  by
  H\"older inequality on each element and over the mesh,
  and the first (interface) term  by
  $\|\mu\|_{Y'} \|v\|_Y$,
  the inequality~\eqref{eq:continuity} follows.     \qedhere
\end{proof}

In particular, inequality~\eqref{eq:continuity} implies  $\|\cl w\|_{\Enrm} \le 2^{1/q} \|\cl w\|_X$, so
\eqref{eq:quasiopt-E} of Theorem~\ref{thm:quasiopt} also gives
\begin{equation}
  \label{eq:quasiopt-X}
  \| \cl u - \cl u_h \|_{\Enrm}
  \,\le\, 2^{1/q}\, ( 1 + 2 C_\varPi )
  \inf_{\cl w_h \in X_h} \| \cl u - \cl w_h \|_X .
\end{equation}
Our convergence rates are in terms of  $h := \max_{K \in \oh} h_K$. We will
use the element-wise $L_2$-orthogonal projection onto $P_j(\oh)$
extended to $L_1$, namely
let $Q_j: L_1(K) \to P_j(K)$ be defined by these moment conditions on each $K \in \oh$ for any $v \in L_1(K)$: 
\begin{equation}
  \label{eq:Q-proj-L1}
  (Q_j v, r)_K = (v, r)_K, \qquad r \in P_j(K).
\end{equation}

\begin{corollary}[Convergence rates]
  \label{cor:rates}
  In addition to the hypotheses of Theorem~\ref{thm:quasiopt},
  suppose the solution $u$ of \eqref{eq:advection} satisfies
  $u \in W_p^s(\Om)$ for an integer $1 \le s \le k+1$. Then,
  \begin{equation*}
    \| \cl u - \cl u_h \|_{\Enrm}
    \,\le\, C\, h^{\min(s, k)}\, \| u \|_{W_p^s(\Om)} ,
  \end{equation*}
  with $C = C(N, k, s, p, \theta_0, \kappa_0)$, independent of
  $\theta_0$ when $k = 1$.
\end{corollary}
\begin{proof}
  By \eqref{eq:quasiopt-X} it suffices to exhibit one
  $\cl w_h \in X_h$ with
  $\| \cl u - \cl w_h \|_X \le C h^{\min(s,k)}
  \|u\|_{W_p^s(\Om)}$. Take
  $\cl w_h := \big( Q_{k-1} u,\; \Dh(I_h u)
    \big),
  $
  where $Q_{k-1}$ is as in~\eqref{eq:Q-proj-L1}
  and $I_h$ is a Scott--Zhang
  interpolation operator \cite{ScottZhang90} of degree $k$, chosen
  so as to preserve the homogeneous boundary values on
  $\Gamma_{\mathrm{in}}$. Write $\varsigma_0 := \min(s,k)$ and
  $\varsigma_1 := \min(s, k+1)$.  For the interior variable, the Bramble--Hilbert lemma immediately gives
  \begin{equation}
    \label{eq:field-approx}
    \nrmK{u - Q_{k-1} u}{p}{\Om}
    \,\lesssim\, h^{\varsigma_0} | u |_{W_p^{\varsigma_0}(\Om)} ,
  \end{equation}
  since $P_{k-1}(K) \supseteq P_{\varsigma_0 - 1}(K)$.

  For the trace component, note that
  $z := u - I_h u \in W_{p,0}(\partial_\beta,\Om)$, because
  $u \in W_{p,0}(\partial_\beta,\Om)$
  and $I_h u$ is a continuous piecewise polynomial vanishing on
  $\Gamma_{\mathrm{in}}$. By~\eqref{eq:Dhz-trace-bound} of
  Lemma~\ref{lem:trace-bound} and
  $\nrmK{\beta\cdot\nabla z}{p}{K} \le |\beta_K| \,|z|_{W_p^1(K)}$,
  \begin{equation*}
    \| \Dh z \|_{Y'}^p
    \le \sum_K \Big( \nrmK{z}{p}{K}^p + h_K^p |z|_{W_p^1(K)}^p\Big).
  \end{equation*}
  The simultaneous approximation and stability properties~\cite{ScottZhang90}  of the
  Scott--Zhang operator in $W_p^s$ yield, for
  $1 \le \varsigma_1 \le k+1$, $
    \nrmK{z}{p}{K} + h_K | z |_{W_p^1(K)}
    \le C\, h_K^{\varsigma_1} | u |_{W_p^{\varsigma_1}(\omega_K)} ,
    $
  where $\omega_K$ is the patch of elements meeting $K$; summing
  over $K$ and using the finite overlap of the patches (by shape
  regularity),
  \begin{equation}
    \label{eq:trace-approx}
    \| \Dh u - \Dh( I_h u) \|_{Y'}
    \le C\, h^{\varsigma_1}\, | u |_{W_p^{\varsigma_1}(\Om)} .
  \end{equation}
  Combining \eqref{eq:field-approx} and \eqref{eq:trace-approx}
  with $\varsigma_0 \le \varsigma_1$ completes the proof.
\end{proof}

The energy norm above is mesh-dependent. We now convert this into a convergence rate
in the mesh-independent $L_p(\Om)$ norm. The
conversion costs one power of $h$. Let
\begin{equation}
  \label{eq:bar-beta}
  \bar\beta := \max_K |\beta_K|.
\end{equation}
Some results use quasi-uniform meshes where
there is a $c_{\mathrm{qu}}>0$ such that
$h_K \ge c_{\mathrm{qu}} h$ for all $K\in\oh$ with $h:=\max_K h_K$.

\begin{corollary}[Interior variable error in $L_p$]
  \label{cor:L2}
  Adopt the hypotheses of Corollary~\ref{cor:rates}, write
  $\cl u = (u, D_h u)$ and $\cl u_h = (u_h, D_h z_h)$. Then
  \begin{equation}
    \label{eq:L2-duality}
    \| u - u_h \|_{L_p(\Om)}
    \,\le\,
    \Big( 1 + C_{\mathrm{PF}}(q,\beta)^q \max_{K \in \oh}
    a_K^{-q} \Big)^{1/q}
    \, \| \cl u - \cl u_h \|_{\Enrm} .
  \end{equation}
  In particular, on quasi-uniform mesh families, with $\bar\beta$ as in \eqref{eq:bar-beta},
  \begin{equation*}
    \| u - u_h \|_{L_p(\Om)}
    \,\le\, C \big( 1 + C_{\mathrm{PF}}(q,\beta)\, \bar\beta\, h^{-1}
    \big)\, h^{\min(s,k)} \, \| u \|_{W_p^s(\Om)}.
  \end{equation*}
\end{corollary}
\begin{proof}
  We use a duality argument. Let $g \in L_q(\Om)$ be arbitrary. Applying
  Lemma~\ref{lem:transport}, with exponent
  $q$, to the reversed advection field $-\beta$ (whose inflow
  boundary is $\Gamma_{\mathrm{out}}$, and for which
  Assumption~\ref{asm:fill}(ii) holds with the roles of $z_+$ and
  $z_-$ interchanged), we obtain $z \in \W{q}{\Om}$ with
  $-\beta \cdot \nabla z = g$ in $\Om$  and
  $z = 0$ on $\Gamma_{\mathrm{out}}.$ Then~\eqref{eq:curvedPF} gives
  $\|z\|_{L_q(\Om)} \le C_{\mathrm{PF}}(q,-\beta) \|g\|_{L_q(\Om)}
  = C_{\mathrm{PF}}(q,\beta) \|g\|_{L_q(\Om)}$, the last equality
  because the interchange of $z_+$ and $z_-$ leaves the maximum
  defining $C_{\mathrm{PF}}$ unchanged. Viewing
  $z$ as an element of the broken space $Y$,
  \begin{equation}
    \label{eq:zYbound}
    \| z \|_Y^q
    = \sum_{K} \Big[ a_K^{-q} \nrmK{z}{q}{K}^q
    + \nrmK{\beta\cdot\nabla z}{q}{K}^q \Big]
    \le \Big( 1 + C_{\mathrm{PF}}(q,\beta)^q \max_K
    a_K^{-q} \Big) \, \| g \|_{L_q(\Om)}^q .
  \end{equation}
  The trace component of $\cl u - \cl u_h$ is
  $\Dh(u - z_h)$ for a $z_h \in S_{h,0}^k$. Since  $u - z_h \in W_{p,0}(\partial_\beta,\Om)$ vanishes on  $\Gamma_{\mathrm{in}}$, while
  $z \in \W{q}{\Om}$ vanishes on $\Gamma_{\mathrm{out}}$, and $\beta\cdot n$ vanishes on the remainder of the boundary $\partial\Om$,
  the integration by parts formula~\eqref{eq:int-by-parts-graphspace} yields
  \begin{equation*}
    \ap{\Dh(u - z_h), z}
    = \int_{\partial\Om} (\beta \cdot n) (u - z_h)\, z \, ds
    = 0.
  \end{equation*}
  Hence, by~\eqref{eq:extended-b},
  \begin{equation*}
    \ip{u - u_h}{g}{\Om}
    = - \ip{u - u_h}{\beta\cdot\nabla z}{\Om}
    = \bform( \cl u - \cl u_h,\, z )
    \,\le\, \| \cl u - \cl u_h \|_{\Enrm} \, \| z \|_Y.
  \end{equation*}
  Combining with \eqref{eq:zYbound} and taking the supremum over
  $\|g\|_{L_q(\Om)} = 1$ proves \eqref{eq:L2-duality} by
  $L_p$--$L_q$ duality. The second assertion follows from
  Corollary~\ref{cor:rates}.
\end{proof}

\section{A posteriori error control}
\label{sec:apost}

The quantity minimized in \eqref{eq:method} is computable, and is a
natural a posteriori error estimator:
\begin{equation}
  \label{eq:estimator}
  \eta \,:=\, \big\| \ell - \bform( \cl u_h, \cdot\,) \big\|_{Y_h'} .
\end{equation}
In practice $\eta$ is evaluated through the discrete \emph{residual
  representative} $\varrho_h \in Y_h$ defined by
\begin{equation}
  \label{eq:residual-rep}
  \ap{J_Y \varrho_h, v} = \ell(v) - \bform(\cl u_h, v),
  \qquad \text{for all } v \in Y_h,
\end{equation}
where $J_Y : Y \to Y'$ is the normalized duality map of $(Y, \|\cdot\|_Y)$, which reduces to the linear
Riesz map when $p = 2$. The pair $(\varrho_h, \cl u_h)$ is exactly the
solution of the mixed system equivalent to \eqref{eq:method} (see
\cite[eq.~(4.2)]{MugaTylerZee19}), so $\varrho_h$ is available at no extra
cost. Since normalized duality maps preserve norms, $\eta = \|\varrho_h\|_Y$, and the
estimator localizes into element-wise computable indicators,
\begin{equation}
  \label{eq:indicators}
  \eta^q = \sum_{K \in \oh} \eta_K^q,
  \qquad
  \eta_K := \| \varrho_h \|_{Y(K)} .
\end{equation}
Set $m_k:=\max\{k-2,0\}$. Reliability is measured against the data
oscillation
\begin{equation}
  \label{eq:osc}
  \osc(f) :=
  \bigg( \sum_{K \in \oh} a_K^p\,
  \nrmK{f - Q_{m_k} f}{p}{K}^p \bigg)^{1/p},
\end{equation}
where $Q_{m_k}$ is as in~\eqref{eq:Q-proj-L1} and $a_K=h_K/|\beta_K|$ as before. For $p = 2$, estimates of the kind below were deduced from
the existence of a Fortin operator in \cite{CarstDemkoGopal14} (see
also \cite{DPGacta}). We now provide them  in the present
setting for all $1 < p < \infty$, with efficiency constant one.

\begin{theorem}[Global reliability and efficiency]
  \label{thm:apost}
  Under the hypotheses of Theorem~\ref{thm:quasiopt}, suppose in
  addition that $V_h^{++}\subseteq Y_h$. Then
  \begin{equation*}
    \eta    \;\le\; \| \cl u - \cl u_h \|_{\Enrm}
    \; \le\; \| \varPi^{++}_h\| \, \eta
    \; + \;
    \| I - \varPi^{++}_h\| \, \osc(f)
  \end{equation*}
  where the operator norms $ \| \varPi^{++}_h\|$ and
  $\| I - \varPi^{++}_h\|$ are bounded mesh-independently  by
  $C_{\varPi^{++}}$ and $1 + C_{\varPi^{++}}$, respectively.
  In particular, $\osc(f)$ vanishes if $f$ is piecewise
  constant.
\end{theorem}

\begin{proof}
  Efficiency follows trivially by \eqref{eq:consistency} and~\eqref{eq:Eh-le-E},
  $\eta = \|\cl u - \cl u_h\|_{\Enrm,h} \le \|\cl u - \cl u_h\|_{\Enrm}$.

  To prove reliability, let $\delta := \cl u - \cl u_h$ and $v \in Y$,
  and split
  $\bform(\delta, v) = \bform(\delta, \varPi_h^{++}v)
  + \bform(\delta, v - \varPi_h^{++}v)$.
  Since $\varPi_h^{++}v \in V_h^{++} \subseteq Y_h$,
  Corollary~\ref{cor:fortin-Y} gives
  \begin{equation}
    \label{eq:apost-fortin-part}
    \bform( \delta, \varPi_h^{++}v )
    = \ell( \varPi_h^{++}v ) - \bform( \cl u_h, \varPi_h^{++}v )
    \,\le\, \eta\, \| \varPi_h^{++}v \|_Y
    \,\le\, \eta\,    \| \varPi_h^{++} \| \, \| v \|_Y.
  \end{equation}
  For the second term, we start by applying the global Fortin property of
  Theorem~\ref{thm:kone-plusplus} to $\cl u_h \in X_h$,
  followed by consistency:
  \begin{align}
    \nonumber
    \bform( \delta, v - \varPi_h^{++}v )
    & = \bform( \cl u, v - \varPi_h^{++}v )
    = \ell( v - \varPi_h^{++}v )
    = \sum_{K \in \oh} \ip{f}{v - \varPi_K^{++}v}{K}
    \\ \nonumber
    & = \sum_{K} \ip{f-Q_{m_k}f}{v-\varPi_K^{++}v}{K}
    \\ \nonumber
    &\le \sum_K \big(a_K\nrmK{f-Q_{m_k}f}{p}{K}\big)
    \big(a_K^{-1}\nrmK{v-\varPi_K^{++}v}{q}{K}\big)
    \\ \label{eq:apost-osc-part}
    &\le \osc(f)\,\|I-\varPi_h^{++}\|\,\|v\|_Y.
  \end{align}
  In the second line, we have used \eqref{eq:F2} and \eqref{eq:F3},
  by which $v-\varPi_K^{++}v$ is $L_2(K)$-orthogonal to $P_{m_k}(K)$, allowing us to replace
  $f$ by $f-Q_{m_k}f$ in each summand. Afterward, we used  H\"older
  inequalities on each element and over the mesh, together with
  $a_K^{-1}\nrmK{\cdot}{q}{K}\le\|\cdot\|_{Y(K)}$.
  Adding the
  bounds~\eqref{eq:apost-fortin-part} and~\eqref{eq:apost-osc-part}, dividing by $\|v\|_Y$, and taking the
  supremum over $0 \neq v \in Y$ completes the proof.
  The bounds
  with
  $ C_{\varPi^{++}}$ follow from~\eqref{eq:piplusplus-Y-bound}.
\end{proof}

It is well known that the $\varrho_h$
solving~\eqref{eq:residual-rep} also
satisfies $\bform(\cl w_h,\varrho_h)=0$ for all $\cl w_h\in X_h$
(which is the second equation of equivalent mixed system of the
practical DPG method: see \cite[eq.~(6.7)]{DPGacta} or
\cite[eq.~(4.2b)]{MugaTylerZee19}), i.e., $\varrho_h$
is in the {\em residual representation space}
\[
    Y_h^0:=\{v\in Y_h:\bform(\cl w_h,v)=0
      \text{ for all }\cl w_h\in X_h\}.
\]
If $Y_h^0$ were trivial, then $\eta\equiv0$ would carry no
information, but this is not the case:

\begin{proposition}
  \label{prop:dimYh0}
  Under the hypotheses of Theorem~\ref{thm:apost},
  when $N \ge 2$, the dimension of the residual representation space $Y_h^0$ is not smaller than the number of mesh elements $N_{\mathrm{el}}$.
\end{proposition}
\begin{proof}
  Define $B_h:X_h\to Y_h'$ by
  $
    (B_h\cl w_h)(v_h):=\bform(\cl w_h,v_h),
    $
    for any $\cl w_h \in X_h, v_h \in Y_h$.
  This map is injective. Indeed, if $B_h\cl w_h=0$, then for every
  $v\in Y$, the Fortin property \eqref{eq:fortin-Xh} gives
  $
    \bform(\cl w_h,v)=\bform(\cl w_h,\varPi_hv)=0,
  $
  because $\varPi_hv\in V_h\subseteq V_h^{++}\subseteq Y_h$.
  Lemma~\ref{lem:injectivity} therefore gives $\cl w_h=0$.
  By rank-nullity theorem, rank of $B_h$ must therefore equal $\dim X_h$.
  The transpose $B_h^*:Y_h\to X_h'$ has the same rank as $B_h$, and
  its kernel is $Y_h^0$, so rank-nullity theorem  gives
  \begin{equation}
    \label{eq:dimYh0}
    \dim Y_h^0=\dim Y_h-\dim X_h.
  \end{equation}
  The same bilinear form also generates an injective map from
  $X_h\to V_h'$ by the same Fortin property and by the same argument
  as above. Hence $ \dim X_h$ cannot be larger than $\dim V_h, $ so \eqref{eq:dimYh0} implies
  \begin{equation}
    \label{eq:dimYh0-inequality}
    \dim Y_h^0 \ge \dim Y_h-\dim V_h.
  \end{equation}

  Finally,
  since $N\ge2$, Theorem~\ref{thm:kone}(a) gives the direct sum
  $V^+(K)=V(K)\oplus P_0(K)$. Because
  $V^+(K)\subseteq V^{++}(K)$, it follows that
  $
    \dim V_h^{++}
    \ge
    \dim V_h+N_{\mathrm{el}}.
  $
  This provides a lower bound on the dimension of $Y_h$ since  $V_h^{++}\subseteq Y_h$.  The  inequality~\eqref{eq:dimYh0-inequality} then implies
  $
    \dim Y_h^0
    \ge \dim V_h^{++}-\dim V_h
    \ge N_{\mathrm{el}}.
  $
\end{proof}

\section{Lowest order special case}
\label{sec:lowest}

In this section, we show how a change of the
test space metric by a scale $\sigma$  yields
optimal first-order rate in $L_2(\Om)$ for the lowest-order case
(improving upon Corollary~\ref{cor:L2}'s rate).
The analysis of this section is the first, to our knowledge, to show
the influence of such a scale~$\sigma$ on  the practical DPG  method's error.
The key ingredients  are   the
augmented Fortin  operator $\varPi_h^{++}$ and
a characterization from~\cite{DPGacta} of how the energy norm changes with $\sigma$.
The latter result is only for Hilbert
spaces, so for brevity, without attempting to extend it here,
we will restrict to $p=q=2$ in this section.

For $0<\sigma<\infty$, the modified norms are given by
\begin{equation}
  \label{eq:Ysigma}
  \| v \|_{Y,\sigma}^2
  := \sigma^{-2} \nrmK{v}{2}{\Om}^2
  + \sum_{K \in \oh} \nrmK{\beta\cdot\nabla v}{2}{K}^2 ,
  \qquad
  \| \cl w \|_{\Enrm,\sigma}
  := \sup_{0 \neq v \in Y} \frac{\bform(\cl w, v)}{\|v\|_{Y,\sigma}},
\end{equation}
for all $v \in Y:= \W{2}{\oh}$ and
$\cl w \in X = L_2(\Om) \times \hat X_0$ where
$\hat X_0 = D_h(W_{2,0}(\partial_\beta,\Om))$ as in~\eqref{eq:Xdef}.  Clearly,
the second norm in~\eqref{eq:Ysigma} is precisely the energy norm
\eqref{eq:energy} when the original test norm is replaced by the new
test norm $\|\cdot\|_{Y,\sigma}$. The subspace $Y_h$ is now considered
with the norm $\| \cdot \|_{Y, \sigma}$ and discrete dual norm
in~\eqref{eq:dualnorm} with $\|\cdot \|_Y$ replaced by
$\| \cdot \|_{Y, \sigma}$ is denoted by
$\| \cdot \|_{Y'_{h, \sigma}}$. In this section, we study
the solution $\clush = (\uhs, \muhs)$ of the residual minimization~\eqref{eq:method}
and the estimator $\etas$ of~\eqref{eq:estimator}, both obtained using
$\| \cdot \|_{Y'_{h, \sigma}}$ in place of $\| \cdot \|_{Y'_h}$.
The effect of such a DPG scale was first analyzed in \cite{GopalMugaOliva14} for the  Helmholtz equation, and later, more generally in~\cite{DPGacta}, both  only in the context of ideal DPG methods.
With our augmented Fortin operator, we can now  give the first analysis of the practical DPG method with the scale.

For $\mu\in\hat X_0$, let $\ET\mu \in W_{2,0}(\partial_\beta,\Om)$  minimize
$\nrmK{\beta\cdot\nabla z}{2}{\Om}$ over all
$z\in W_{2,0}(\partial_\beta,\Om)$ satisfying $\Dh z=\mu$. The
theory of \cite[Theorem~7.9 and
Propositions~7.10--7.11]{DPGacta}, whose bijectivity hypothesis follows
from Lemma~\ref{lem:transport}, shows that for all $\cl w = (w, \mu)$ in $X$,
\begin{equation}
  \label{eq:Xsigma-equiv}
  (1+k_\sigma)^{-1} \| \cl w \|_{\Enrm,\sigma}^2
  \le \nrmK{w}{2}{\Om}^2
  + \sigma^2 \nrmK{\beta\cdot\nabla \ET\mu}{2}{\Om}^2
  \le (1+k_\sigma) \| \cl w \|_{\Enrm,\sigma}^2 ,
\end{equation}
with $k_\sigma = \tfrac12(c_\sigma + \sqrt{c_\sigma^2 + 4c_\sigma})$
and $c_\sigma=C_{\mathrm{PF}}(2,\beta)^2/\sigma^2$. In particular, for every
$\sigma>0$,
\begin{equation}
  \label{eq:L2-scaled-energy}
  \nrmK{w}{2}{\Om}
  \le \sqrt{1+k_\sigma}\,\|(w,\mu)\|_{\Enrm,\sigma}.
\end{equation}
Write
$\oscs(f):=\sigma\nrmK{f-Q_0f}{2}{\Om}.$
Recall that  $a_K := h_K/|\beta_K|$, that $\bar\beta$ is as in \eqref{eq:bar-beta},
and that $Q_0$ is the
element-wise $L_2(\Om)$-orthogonal projection onto $P_0(\oh)$. In
the next theorem,
no non-characteristic assumption is needed.

\begin{theorem}[A priori and a posteriori estimates]
  \label{thm:kone-rates}
  Let $k = 1$ and $p = 2$, suppose Assumptions~\ref{asm:beta}
  and~\ref{asm:fill} and \eqref{eq:shapereg} hold, and that
  $V_h^{++} \subseteq Y_h$ with $Y_h$
  satisfying~\eqref{eq:Yh-condition}. Fix any $\sigma>0$
  and consider meshes with $\max_K a_K \le \sigma$.  Then, there are
  constants $C_i=C_i(N,\kappa_0)$, $i=1,2$, independent of $\beta$, $\sigma$,
  and the mesh, such that the following statements hold.
  \begin{enumerate}[label=(\alph*)]
  \item The estimator $\etas$ is efficient and reliable up to data
    oscillation,
    \begin{equation}
      \label{eq:kone-apost}
      \etas\le\|\cl u -\clush\|_{\Enrm,\sigma}
      \le C_1\etas
        +(1+C_1)\oscs(f),
    \end{equation}
    and is also a reliable indicator of the error in the interior variable,
    \begin{equation}
      \label{eq:kone-L2-apost}
      \nrmK{u-\uhs}{2}{\Om}
      \le\sqrt{1+k_\sigma}\big[
        C_1\etas
        +(1+C_1)\oscs(f)\big].
    \end{equation}
    If $f$ is piecewise constant on the mesh,  then $\oscs(f)=0$.

  \item If $u \in H^2(\Om)$, then with $\bar\beta$ as in \eqref{eq:bar-beta},
    \begin{equation}
      \label{eq:kone-L2-rate}
      \nrmK{u-\uhs}{2}{\Om}
      \,\le\, C_2(1+k_\sigma)\big(1+\sigma\bar\beta\big)\,h\,
      \| u \|_{H^2(\Om)}.
    \end{equation}
  \end{enumerate}
\end{theorem}

\begin{proof}
  Theorem~\ref{thm:kone-plusplus} shows that $\varPi_h^{++}$ maps into
  $V_h^{++}$, has the Fortin property, and, by \eqref{eq:kone-sigma},
  with $\sigma_K=\sigma\ge a_K$, satisfies
  $ \|\varPi_h^{++}v\|_{Y,\sigma} \le C_1\|v\|_{Y,\sigma},$ with a
  constant $C_1$ depending only on $N$ and $\kappa_0$.

  \emph{Proof of (a).}
  Consistency~\eqref{eq:consistency} and
  the inclusion $Y_h\subset Y$ immediately give the efficiency bound
  $\etas\le\|\cl{u} -\clush\|_{\Enrm,\sigma}$ as in the proof of Theorem~\ref{thm:apost}.
  For reliability, repeating the same
  splitting in that proof, but now with
  $\varPi_h^{++}$, its Fortin property eliminates
  $\bform(\clush,v-\varPi_h^{++}v)$, while its mean-preserving property
  gives, on every element,
  \begin{equation*}
    \ip{f}{v-\varPi_K^{++}v}{K}
    =\ip{f-Q_0f}{v-\varPi_K^{++}v}{K}.
  \end{equation*}
  The Cauchy--Schwarz inequality, first element-wise and then over the
  mesh, yields
  \begin{equation*}
    \big|\ip{f}{v-\varPi_h^{++}v}{\Om}\big|
    \le\oscs(f)\,
      \|(I-\varPi_h^{++})v\|_{Y,\sigma}
    \le(1+C_1)\oscs(f)\|v\|_{Y,\sigma}.
  \end{equation*}
  Taking the supremum over $v$ proves \eqref{eq:kone-apost}.
  Applying \eqref{eq:L2-scaled-energy} to
  $\cl u -\clush$ gives \eqref{eq:kone-L2-apost} and its
  piecewise-constant specialization.

  \emph{Proof of (b).}
  The proof of Theorem~\ref{thm:quasiopt}, with the scaled test and
  energy norms, gives
  \begin{equation}
    \label{eq:kone-scaled-quasiopt}
    \|\cl{u}-\clush\|_{\Enrm,\sigma}
    \;\lesssim\; \inf_{\cl w_h\in X_h}
      \|\cl{u}-\cl w_h\|_{\Enrm,\sigma}.
  \end{equation}
  Take $\cl w_h=(Q_0u,\Dh I_hu)$, where $I_h$ is the degree-one
  Scott--Zhang interpolant preserving the inflow boundary condition.
  Since $u-I_hu$ is an admissible extension of
  $\Dh(u-I_hu)$, standard approximation estimates and
  \eqref{eq:Xsigma-equiv} give
  \begin{align*}
    \|\cl{u}-\cl w_h\|_{\Enrm,\sigma}
    &\le \sqrt{1+k_\sigma}\left(
      \nrmK{u-Q_0u}{2}{\Om}
      +\sigma\nrmK{\beta\cdot\nabla(u-I_hu)}{2}{\Om}
      \right)\\
    &\le C\sqrt{1+k_\sigma}(1+\sigma\bar\beta)
      h\|u\|_{H^2(\Om)}.
  \end{align*}
  Combining this with \eqref{eq:kone-scaled-quasiopt} and
  \eqref{eq:L2-scaled-energy} proves~\eqref{eq:kone-L2-rate}.
\end{proof}

\begin{remark}[Optimal $\sigma$]
  \label{rem:sigma-choice}
  The two factors in \eqref{eq:kone-L2-rate} pull in opposite
  directions: 
  $1 + k_\sigma$ decreases to $1$ while $1 + \sigma\bar\beta$
  grows linearly as $\sigma \to \infty$. As $\sigma\to 0$,
  an elementary expansion shows that 
  $k_\sigma =  (C_{\mathrm{PF}}(2,\beta)/\sigma)^2  + 1 +  O(\sigma^{2})$.  
  The product
  $(1+k_\sigma)\big(1+\sigma\bar\beta\big)$
  therefore tends to infinity at both
  ends so must have a minimum
  determined by $C_{\mathrm{PF}}(2,\beta)$ and $\bar\beta$ alone. Also observe that   it is essential to hold  $\sigma$ fixed as $h \to 0$ (so that the 
  hypothesis $\max_K a_K \le \sigma$ of Theorem~\ref{thm:kone-rates}
  is satisfied as $h \to 0$.) The smallest admissible choice
  $\sigma = \max_K a_K$ destroys the rate (as then $\sigma \eqsim h$,
  $k_\sigma \eqsim h^{-2}$, and the right hand side of
  \eqref{eq:kone-L2-rate} diverges like $h^{-1}$). The optimal 
  rate is thus a consequence of decoupling the test norm scale from
  the mesh scale.
\end{remark}

\section{Closing remarks}
\label{sec:remarks}

Discrete stability analysis of practical DPG methods is more involved for transport equations than for second-order elliptic problems, for which  enriched polynomial test spaces are available~\cite{GopalQiu14,NagarPetriDemko17}.
Ever since~\cite{DemkoGopal10} exhibited nonpolynomial optimal test functions for advection,
finding a stable polynomial test space surrogate has remained a challenge.
The key difficulty is to  obtain estimates uniform in the direction of the flow relative to the mesh.
Until now, the only DPG stability result for advection,
using standard finite element spaces,
that was free
of any condition on the flow direction was that
of~\cite{BroerDahmeSteve18}.
They traced the difficulty to the instability of optimal test functions as $\beta$ becomes tangent to an element facet, and recovered orientation-uniform stability using polynomial surrogates on a sufficiently deep fixed subgrid \cite[Section~4.4 and Theorem~4.8]{BroerDahmeSteve18}. Because the required refinement depth is not quantified, their result left open whether an explicit polynomial test space on the original mesh can suffice. We have answered this question affirmatively without any directional restrictions for $k=1$, and under a uniform non-characteristic facet condition for $k\ge 2$.

Our Fortin construction  for the
ultraweak advection formulation  splits into  facet conditions
involving the signed facet-bubble combination
$\bdr$, and interior moment conditions involving the
interior bubble $\bint$. Building the local test space out of exactly
these two ingredients makes the Fortin system square and
unisolvent, yielding a minimal test space on the same mesh (without any refinement).
Everything quantitative then follows from the scaling estimates of
Section~\ref{sec:scaling}, in which the advection vector is carried
as a variable through a compactness argument, so that the constants are uniform over admissible flow
directions and invariant under rescaling of $\beta$.

At lowest order $k=1$, our results have no directional restrictions and admit meshes containing characteristic facets.
In our first step, the minimal test space, without any augmentation, gave
the vacuous rate $h^{\min(s,k)-1}=1$ for the field in the mesh-independent norm.
Adding one constant per element and choosing a fixed scale $\sigma\ge\max_K a_K$ recovered the optimal first-order rate in~$L_2(\Om)$, and
an  additional mean-preserving bubble augmentation yielded  an a posteriori bound whose data oscillation vanishes for piecewise constant data.
That so
small a modification should close the gap is
striking. It  suggests that the
scale $\sigma$ deserves to be treated as a genuine design parameter
whose numerical and adaptive selection merits further study.
For $k\ge 2$, the result is qualified (needing
Assumption~\ref{asm:theta}).

Finding a test space with this qualification removed
is an open  problem.
The Fortin conditions themselves are not the obstruction: the correction formula \eqref{eq:piplus} remains Fortin for every linear map $R_K$. The missing step is to choose $R_K$ so that $v-R_Kv$ annihilates the data in~\eqref{eq:annihilated}, allowing the refined bounds \eqref{eq:varPi-refined} to apply. The mean projection has this property for $k=1$. For $k=2$, it annihilates the interior data $P_0(K)$, but not the generally nonconstant boundary data $\beta\cdot\nabla\Pkperp{2}(K)$; for $k\ge3$, it generally annihilates neither.
This suggests seeking a bounded linear map
$
R_K:\W{q}{K}\rightarrow
Z_{k'}(K):=\{\varphi\in P_{k'}(K):
\beta_K\cdot\nabla\varphi=0\}
$
satisfying
$
\ip{v-R_Kv}{\mu}{K}=0
$ for  all $\mu\in M_k(K):=P_{k-2}(K)+
\beta\cdot\nabla\Pkperp{k}(K)
\subseteq P_{k-1}(K).
$
Then $v-R_Kv$ satisfies \eqref{eq:annihilated}, and the proof of Theorem~\ref{thm:kone}(b) extends with $V(K)+Z_{k'}(K)$ in place of $V(K)+P_0(K)$. (No projection property is required.)
The open question is whether such an $R_K$ can be constructed with uniformly bounded degree and norm over the admissible flow directions. The studies in~\cite{DemkoRoberMunoz22} suggest $k'\ge2k-2$ in two dimensions, while the optimal test space of \cite{DemkoGopal10} supplies the corresponding nonpolynomial streamline-constant correction. One other limitation of this work
is that  $\beta$ is assumed piecewise constant. Extension of the Fortin analysis to variable fields is desirable and is postponed to future work.

Ongoing work aims to apply the results here in a
scientific machine learning context, where we approximate the
high-dimensional parameter-to-solution map $\beta \mapsto u$
of~\eqref{eq:advection} by neural networks (NNs). Traditional
applications of Theorem~\ref{thm:apost}
and~\eqref{eq:kone-apost} are in adaptivity, but as explained in~\cite{CastiDahmeGopal25}, they also generate variationally correct loss functions
for training NNs. The bound~\eqref{eq:kone-apost} for $\etas$ is
particularly interesting in this context: for $k = 1$ it is robust in
the element-wise directions of $\beta$ on every shape-regular mesh
with $\max_K a_K \le \sigma$. Moreover, $\sigma$ may be supplied to
the NN as an input parameter alongside the element-wise~$\beta$.
Because of this ongoing numerical work, and ample
prior numerical studies on  DPG methods for advection~\cite{DemkoGopal10,DemkoGopal11,DemkoRoberMunoz22},
no further numerical experiments are reported here.

\section*{Funding}

This work was supported in part by NSF Grant 2245077.

\bibliographystyle{siam}
\bibliography{advection_dpg}

\begin{thebibliography}{10}

\bibitem{BroerDahmeSteve18}
{\sc D.~Broersen, W.~Dahmen, and R.~P. Stevenson}, {\em On the stability of
  {DPG} formulations of transport equations}, Math. Comp., 87 (2018),
  pp.~1051--1082.

\bibitem{Canti17}
{\sc P.~Cantin}, {\em Well-posedness of the scalar and the vector
  advection--reaction problems in {Banach} graph spaces}, C. R. Acad. Sci.
  Paris, Ser. I, 355 (2017), pp.~892--902.

\bibitem{CarstDemkoGopal14}
{\sc C.~Carstensen, L.~Demkowicz, and J.~Gopalakrishnan}, {\em A posteriori
  error control for {DPG} methods}, SIAM J. Numer. Anal., 52 (2014),
  pp.~1335--1353.

\bibitem{CarstDemkoGopal16}
{\sc C.~Carstensen, L.~Demkowicz, and J.~Gopalakrishnan}, {\em Breaking spaces
  and forms for the {DPG} method and applications including {Maxwell}
  equations}, Comput. Math. Appl., 72 (2016), pp.~494--522.

\bibitem{CastiDahmeGopal25}
{\sc P.~{Cort{\'e}s Castillo}, W.~Dahmen, and J.~Gopalakrishnan}, {\em {DPG}
  loss functions for learning parameter-to-solution maps by neural networks},
  ArXiv Preprint,  (2025).

\bibitem{DahmeHuangSchwa11a}
{\sc W.~Dahmen, C.~Huang, C.~Schwab, and G.~Welper}, {\em Adaptive
  {Petrov-Galerkin} methods for first order transport equations}, SIAM J.
  Numer. Anal., 50 (2012), pp.~2420--2445.

\bibitem{DemkoGopal10}
{\sc L.~Demkowicz and J.~Gopalakrishnan}, {\em A class of discontinuous
  {P}etrov--{G}alerkin methods. {Part I}: {T}he transport equation}, Comput.
  Methods Appl. Mech. Engrg., 199 (2010), pp.~1558--1572.

\bibitem{DemkoGopal11}
\leavevmode\vrule height 2pt depth -1.6pt width 23pt, {\em A class of
  discontinuous {P}etrov--{G}alerkin methods. {II}. {O}ptimal test functions},
  Numer. Methods Partial Differential Equations, 27 (2011), pp.~70--105.

\bibitem{DPGacta}
{\sc L.~Demkowicz and J.~Gopalakrishnan}, {\em The discontinuous
  {Petrov}--{Galerkin} method}, Acta Numerica, 34 (2025), pp.~293--384.

\bibitem{DemkoRoberMunoz22}
{\sc L.~Demkowicz, N.~V. Roberts, and J.~Mu{\~n}oz-Matute}, {\em The {DPG}
  method for the convection-reaction problem, revisited}, Comput. Methods Appl.
  Math., 23 (2023), pp.~93--125.

\bibitem{FuhreHeuer24}
{\sc T.~F\"{u}hrer and N.~Heuer}, {\em Robust {DPG} test spaces and {F}ortin
  operators---{T}he {$H^1$} and {$H(\mathrm{div})$} cases}, SIAM Journal on
  Numerical Analysis, 62 (2024), pp.~718--748.

\bibitem{GopalMonkSepul15}
{\sc J.~Gopalakrishnan, P.~Monk, and P.~Sep\'ulveda}, {\em A tent pitching
  scheme motivated by {Friedrichs} theory}, Computers and Mathematics with
  Applications, 70 (2015), pp.~1114--1135.

\bibitem{GopalMugaOliva14}
{\sc J.~Gopalakrishnan, I.~Muga, and N.~Olivares}, {\em Dispersive and
  dissipative errors in the {DPG} method with scaled norms for the {Helmholtz}
  equation}, SIAM J. Sci. Comput., 36 (2014), pp.~A20--A39.

\bibitem{GopalQiu14}
{\sc J.~Gopalakrishnan and W.~Qiu}, {\em An analysis of the practical {DPG}
  method}, Math. Comp., 83 (2014), pp.~537--552.

\bibitem{Guerm99}
{\sc J.-L. Guermond}, {\em Stabilization of {Galerkin} approximations of
  transport equations by subgrid modeling}, M2AN Math. Model. Numer. Anal., 33
  (1999), pp.~1293--1316.

\bibitem{Jense04}
{\sc M.~Jensen}, {\em Discontinuous {Galerkin} Methods for {Friedrichs} Systems
  with {Irregular} Solutions}, PhD thesis, University of Oxford, 2004.

\bibitem{MugaTylerZee19}
{\sc I.~Muga, M.~J.~W. Tyler, and K.~G. van~der Zee}, {\em The discrete-dual
  minimal-residual method ({DDMRes}) for weak advection-reaction problems in
  {Banach} spaces}, Comput. Methods Appl. Math., 19 (2019), pp.~557--579.

\bibitem{MugaZee20}
{\sc I.~Muga and K.~G. van~der Zee}, {\em Discretization of linear problems in
  {Banach} spaces: residual minimization, nonlinear {Petrov}--{Galerkin}, and
  monotone mixed methods}, SIAM J. Numer. Anal., 58 (2020), pp.~3406--3426.

\bibitem{NagarPetriDemko17}
{\sc S.~Nagaraj, S.~Petrides, and L.~Demkowicz}, {\em Construction of {DPG}
  {F}ortin operators for second order problems}, Comput. Math. Appl., 74
  (2017), pp.~1964--1980.

\bibitem{ScottZhang90}
{\sc L.~R. Scott and S.~Zhang}, {\em Finite element interpolation of nonsmooth
  functions satisfying boundary conditions}, Math. Comp., 54 (1990),
  pp.~483--493.

\end{thebibliography}

\end{document}